\documentclass[11pt]{article}
\usepackage{fullpage}
\usepackage{times}
\usepackage{latexsym,amsfonts,amsmath}
\usepackage{graphics,epic,eepic,color}
\usepackage{ifpdf}
\usepackage{bbm}
\usepackage{algorithm,algorithmicx,algpseudocode}

\newtheorem{theorem}{Theorem}

\newtheorem{proposition}[theorem]{Proposition}
\newtheorem{corollary}[theorem]{Corollary}
\newtheorem{lemma}[theorem]{Lemma}

\newtheorem{definition}[theorem]{Definition}

\newtheorem{remark}[theorem]{Remark}
\newenvironment{proof}{\noindent{\bf Proof:}\ }{\(\qed\) \par\medskip}

\newcommand{\qed}{\quad\mbox{\rule{7pt}{7pt}}}

\newcommand{\cmt}[1]{\ifhmode\newline\fi{\sf *** \ \ #1 \\}}

\newcommand{\NN}{\mathbf{N}}

\newcommand{\RR}{\mathbf{R}}
\newcommand{\indic}{\mathbbm{1}}

\newcommand{\calA}{\mathcal{A}}
\newcommand{\calD}{\mathcal{D}}

\newcommand{\calF}{\mathcal{F}}
\newcommand{\calG}{\mathcal{G}}
\newcommand{\calI}{\mathcal{I}}
\newcommand{\calO}{\mathcal{O}}
\newcommand{\calP}{\mathcal{P}}

\newcommand{\calS}{\mathcal{S}}
\newcommand{\calT}{\mathcal{T}}

\newcommand{\mediannopar}{{\textstyle\mathop{{\sf Med}}}}
\newcommand{\median}[2]{{\textstyle\mathop{{\sf Med}}_{#1}}\left[{#2}\right]}

\newcommand{\set}[2]{\{\,#1\,|\,#2\,\}}

\newcommand{\setof}[1]{\left\{#1\right\}}

\newlength{\first}\newlength{\second}

\newcommand{\eqdef}{\stackrel{{\rm def}}{=}}

\newcommand{\smidge}{{\kern .05em}}

\newlength{\subjtolt}

\newcommand{\prob}[2]{{\textstyle{\mathop{{\sf Pr}}_{#1}}}\left[\, #2\,\right]}
\newcommand{\expec}[2]{{\textstyle{\mathop{{\sf E}}_{#1}}}\left[\, #2\,\right]}
\newcommand{\var}[2]{{\textstyle{\mathop{{\sf V}}_{#1}}}\left[\, #2\,\right]}

\newlength{\saveparindent}
\def\bproof{\begin{rm}\protect\vspace{5pt}\noindent{\bf Proof: }%
\addtolength{\parskip}{4pt}\setlength{\parindent}{0pt}}
\def\eproof{\end{rm}\addtolength{\parskip}{-4pt}%
\setlength{\parindent}{\saveparindent}}

\newcommand{\bprooff}[1]{\begin{rm}\protect\vspace{5pt}%
\noindent{\bf Proof of #1: }\addtolength{\parskip}{4pt}%
\setlength{\parindent}{0pt}}

\newenvironment{prooff}[1]{\par\bprooff{#1}}{\eproof\qed\par}

\begin{document}
\title{Generalizations and Variants of the Largest Non-crossing Matching Problem in Random Bipartite Graphs}

\author{
  Marcos Kiwi\thanks{
  Depto.~Ing.~Matem\'{a}tica \&
  Ctr.~Modelamiento Matem\'atico UMI 2807, U.~Chile.
  Web: \texttt{www.dim.uchile.cl/$\sim$mkiwi}.
  Gratefully acknowledges the support of
    CONICYT via Basal in Applied Mathematics and
    FONDECYT 1090227.}
  \and
Jos\'e A.~Soto\thanks{Department of Mathematics,
  MIT.
  \texttt{jsoto@math.mit.edu} .
 Gratefully acknowledges support from NSF contract CCF-0829878 and ONR
grant N00014-11-1-0053.}
}

\maketitle
\begin{abstract}
A two-rowed array $\alpha_{n}=\begin{pmatrix}a_1 & a_2 & \ldots & a_n \\
  b_1 & b_2 & \ldots & b_n \end{pmatrix}$ is said to be in lexicographic
  order if $a_k\leq a_{k+1}$ and 
  $b_{k}\leq b_{k+1}$ if $a_{k}=a_{k+1}$.
A length $\ell$ (strictly) increasing subsequence of $\alpha_{n}$ is a set of 
  indices $i_1 < i_2 < \ldots < i_{\ell}$ such that 
  $b_{i_1} < b_{i_2} < \ldots < b_{i_\ell}$.
We are interested in the statistics of the length of the 
  longest increasing subsequence of $\alpha_{n}$
  chosen according to $\calD_{n}$, for distinct families of 
  distributions $\calD=(\calD_{n})_{n\in\NN}$,
  and when $n$ goes to infinity.
This general framework encompasses well studied problems such as the
  so called Longest Increasing Subsequence problem, the Longest 
  Common Subsequence problem, problems concerning directed bond
  percolation models, among others.
We define several natural 
  families of distinct distributions and characterize
  the asymptotic behavior of the length of a longest increasing 
  subsequence chosen according to them.
In particular, 
  we consider generalizations to $d$-rowed arrays as well as symmetry
  restricted two-rowed arrays.
\end{abstract}

\date{}

\section{Introduction}\label{sec:intro}
Suppose that we select uniformly at random a permutation
  $\pi$ of $[n]\eqdef\setof{1,\ldots,n}$.
We can associate to $\pi$ the two-rowed 
  lexicographically sorted array 
  $\alpha_{\pi}=\begin{pmatrix} 1 & 2 & \ldots & n \\ 
  \pi(1) & \pi(2) & \ldots & \pi(n) \end{pmatrix}$.
We denote by $lis(\pi)$ the length of a longest 
  increasing subsequence of $\alpha_{\pi}$.
The determination, as $n\to\infty$, of the first moments
  of $lis(\pi)$ has been a 
  problem of much interest for a long time  
  (for surveys see~\cite{ad99,or98,stanley02} and 
  references therein).
This line of research led to what is considered a 
  major breakthrough: the determination by Baik, Deift and 
  Johansson~\cite{bdj99} of, after proper scaling, the distribution of 
  $lis(\cdot)$.
In~\cite{br99}, variations are studied 
  where instead of permutations of $[n]$,
  random involutions, signed permutations, and signed involutions
  are selected at random. 
Generalizations where $d-1$ random permutations are selected 
  can be restated as problems concerning 
  longest increasing subsequences of $d$-rowed arrays.

Suppose now that we select uniformly at random two words
  $\mu$ and $\nu$ from $\Sigma^{n}$, where 
  $\Sigma$ is some finite alphabet of size $k$.
We can associate to $(\mu,\nu)$ the two-rowed 
  lexicographically sorted array $\alpha_{\mu,\nu}$ 
  where $\begin{pmatrix} i \\ j\end{pmatrix}$ is a 
  column of $\alpha_{\mu,\nu}$ if and only if
  the $i$-th character of $\mu$ is the same as the 
  $j$-th character of $\nu$
  (for an example, see Figure~\ref{fig:example-lcs}).
\begin{figure}[h]\label{fig:example-lcs}
\[
\begin{pmatrix}
1 & 1 & 2 & 2 & 3 & 3 & 4 & 5 & 5 \\
3 & 5 & 1 & 2 & 3 & 5 & 4 & 3 & 5
\end{pmatrix}
\]
\caption{Lexicographically ordered two-rowed array associated 
  to words $abaca$ and $bbaca$}
\end{figure}
The length of a longest common subsequence of $\mu$ and 
  $\nu$, denoted $lcs(\mu,\nu)$, equals the length of a longest
  increasing subsequence of $\alpha_{\mu,\nu}$.
Since the mid 70's, it has been known~\cite{cs75} that 
  the expectation of $lcs(\mu,\nu)$ when normalized
  by $n$, converges to a constant $\gamma_{k}$ (the so called
  Chv\'atal-Sankoff constant).
The determination of the exact value of $\gamma_{k}$, for $k$ fixed, 
  remains a challenging open problem.
To the best of our knowledge, 
  the asymptotic distribution theory of the longest increasing
  subsequence problem is essentially uncharted territory.
Generalizations where $d$ random length $n$ words are 
  chosen from a finite alphabet $\Sigma$ 
  can also be restated as problems concerning 
  longest increasing subsequences of $d$-rowed arrays.

We now discuss yet one more relevant instance,
  previously considered by Sepp\"al\"ainen~\cite{seppalainen97}, 
  and encompassed by the framework described above. 
Fix a parameter $0<p<1$ and let $n$ be a positive integer.
For each site of the lattice $[n]^{2}$, let a point be present
  (the site is occupied) with probability $p$ and absent
  (the site is empty) with probability $q=1-p$,
  independently of all the other sites.
Let $\omega:[n]^{2}\to\setof{0,1}$ be an encoding of the occupied/empty
  sites ($1$ representing an occupied site and $0$ a vacant one).
We can associate to $\omega$ a two-rowed lexicographically sorted array
  $\alpha_{\omega}$ where $\begin{pmatrix} i \\ j\end{pmatrix}$ is a
  column of $\alpha_{\omega}$ if and only if site $(i,j)\in [n]^{2}$
  is occupied.
Let $L(\omega)$ equal the number of sites on a longest 
  strictly increasing path of occupied sites according to $\omega$,
  where a path $(x_1,y_1), (x_2,y_2),\ldots (x_m,y_m)$ of points
  on $[n]^{2}$ is strictly increasing if $x_1<x_2<\ldots <x_m$
  and $y_1<y_2<\ldots <y_m$.
Observe that $L(\omega)$ equals the length of a longest increasing 
  subsequence of $\alpha_{\omega}$.
Subadditivity arguments easily imply that the expected value of 
  $L(\omega)$, when normalized
  by $n$, converges to a constant $\gamma_{p}$.  
Via a reformulation of the problem as one of interacting particle
  systems, Sepp\"al\"ainen~\cite{seppalainen97} shows that 
  $\gamma_{p}=2\sqrt{p}/(1+\sqrt{p})$.
Also worth noting is that the same object $\alpha_{\omega}$ arises
  in the study of the asymptotic shape of a directed bond percolation 
  model (see~\cite[\S 1]{seppalainen97} for details).
Symmetric variants, where for example site $(i,j)$ is occupied
  if and only if $(j,i)$ is occupied, can be easily formulated.
Generalizations where $d$-dimensional lattices are considered
  can also be restated as problems concerning 
  longest increasing subsequences of $d$-rowed arrays.
However, to the best of our knowledge, neither of the latter 
  two variants has been considered in the literature.

\medskip
Thus far, we have described well studied scenarios where 
  the general problem formulated in the abstract naturally arises.
This motivates our work.
However, for the sake of clarity of exposition and in order to use
  more convenient notation, it will be preferable to reformulate
  the issues we are interested in as one concerning hyper-graphs.
To carry out this reformulation, below we introduce some 
  useful terminology and then address in this language 
  the problem of determining the statistics of the length of 
  a longest increasing subsequence of a 
  randomly chosen lexicographically sorted 
  $d$-rowed array.

\medskip
Let $A_1,\ldots,A_d$ be $d$ disjoint (finite) sets, also called 
  \emph{color classes}.
We assume that over each $A_i$ there is a total order relation,
  which abusing notation, we denote $\leq$ in all cases.
When we consider subsets of a totally ordered color class we always
  assume the subset inherits, and thus respects, the original order.
A \emph{$d$-partite hyper-graph} over
  totally ordered color classes $A_1,\ldots,A_d$
  with edge set $E\subseteq A_1\times\ldots\times A_d$ 
  is a tuple $G=(A_1,\ldots,A_d;E)$,
  and its edge set is denoted by $E(G)$.
For $A'_i\subseteq A_i$ with $1\leq i\leq d$ and 
  hyper-graph $G=(A_1,\ldots,A_d;E)$, we denote by 
  $\left. G\right|_{A'_1\times\ldots\times A'_d}$ the \emph{hyper-subgraph
  induced by $G$ in $A'_1\times\ldots\times A'_d$}, i.e.~the
  hyper-graph with node set $V'=A'_1\times\ldots\times A'_d$
  and edge set $E\cap V'$.
We say that two hyper-graphs are \emph{disjoint} if their 
  corresponding vertex sets are disjoint.
Let $K_{A_1,\ldots,A_d}$ denote the \emph{complete $d$-partite hyper-graph} over 
  color classes $A_1,\ldots,A_d$ whose edge set is 
  $A_1\times A_2\times \ldots A_d$.
Henceforth, we denote the cardinality of $A_i$ by $n_i$.
If we identify $A_i$ with $[n_i]$,
  then we write $K_{n_1,\ldots,n_d}$ instead of $K_{A_1,\ldots,A_d}$.
If $n_1=\ldots=n_d$, then we write $K^{(d)}_{n}$ instead of 
  $K_{n_1,\ldots,n_d}$.
Over the edge set of $K_{A_1,\ldots,A_d}$ we consider the natural partial 
  order relation $\preceq$ defined by
\begin{eqnarray*}
(v_1,\ldots,v_d) \preceq (v'_{1},\ldots,v'_{d}) & \Longleftrightarrow & 
  v_{i}\leq v'_{i} \text{ for all $1\leq i\leq d$}\,.
\end{eqnarray*}
We say that a collection of node-disjoint edges 
  $M\subseteq E(G)$ is a \emph{non-crossing hyper-matching}
  if for every pair of edges $e,f\in M$ it holds that
  $e\preceq f$ or $f\preceq e$.
When $G=(A_1,\ldots,A_d;E)$ is such that $E(G)$ is a non-crossing 
  hyper-matching we will simply say that $G$ is a non-crossing 
  $d$-partite hyper-graph, or simply a non-crossing 
  hyper-matching.
Furthermore, we will denote by $L(G)$ the size of a largest 
  non-crossing hyper-matching of $G$ and by $L(\calF)$ the random 
  variable $L(G)$ when $G$ is chosen according to a distribution~$\calF$
  over $d$-partite hyper-graphs.
When we want to stress that we are dealing with only two color classes,
  we will speak of graphs and matchings instead of hyper-graphs and 
  hyper-matchings.

Now, consider a family of distributions $\calD=(\calD(K_{A_1,\ldots,A_d}))$ 
  where each $\calD(K_{A_1,\ldots,A_d})$ is a probability distribution 
  over subgraphs of $K_{A_1,\ldots,A_d}$.
In this work we are interested in understanding what we refer to 
  as the Longest Non-crossing Matching problem, i.e.~the 
  behavior of the expectation of $L(G)$ when~$G$ 
  is chosen according to various distinct families 
  of distributions $\calD=(\calD(K^{(d)}_{n}))$
  and $n$ goes to infinity.
Of course, in order to be able to derive some meaningful results
  we will need some assumptions on the distributions 
  $\calD(K^{(d)}_{n})$.
Below, we encompass in a definition a minimal set of assumptions that
  are both easy to establish and general enough to capture several
  relevant scenarios.
\begin{definition}\label{def:random}
Let $\calD=(\calD(K_{A_1,\ldots,A_d}))$ be a family of distributions where
  each $\calD(K_{A_1,\ldots,A_d})$ is a probability distribution over the 
  collection of hyper-subgraphs of $K_{A_1,\ldots,A_d}$.
We say that $\calD$ is a random $d$-partite hyper-graph model if 
  for $H$ chosen according to $\calD(K_{A_1,\ldots,A_d})$ the following
  two conditions hold:
\begin{enumerate}
\item \textbf{Monotonicity:}
If $A'_i\subseteq A_i$ with $1\leq i\leq d$ and $n'_i=|A_{i}|$,
  then the distribution of $\left. H\right|_{A'_1\times\ldots\times A'_d}$ 
  is $\calD(K_{n'_1,\ldots,n'_d})$.

\item \textbf{Block independence:} 
If $A'_i,A''_i\subseteq A_i$ are disjoint with $1\leq i\leq d$,  
  then $H'=\left. H\right|_{A'_1\times\ldots\times A'_d}$ and
  $H''=\left. H\right|_{A''_1\times\ldots\times A''_d}$ 
  are independent (and so, $L(H')$ and $L(H'')$ are also
  independent).
\end{enumerate}
\end{definition}
For some of the results we will establish, the following 
  weaker notion will suffice.
\begin{definition}\label{def:weak-random}
Let $\calD=(\calD(K^{(d)}_{n}))$ be a family of distributions where
  each $\calD(K^{(d)}_{n})$ is a probability distribution over the 
  collection of hyper-subgraphs of $K^{(d)}_{n}$.
We say that $\calD$ is a weak random $d$-partite hyper-graph model if 
  for $H$ chosen according to $\calD(K^{(d)}_{n})$ the following
  two conditions hold:
\begin{enumerate}
\item \textbf{Weak monotonicity:}
If $A'\subseteq [n]$, $|A'|=n'$,
  then the distribution of $\left. H\right|_{A'\times\ldots\times A'}$ 
  is $\calD(K^{(d)}_{n'})$.

\item \textbf{Weak block independence:} 
If $A',A''\subseteq [n]$ are disjoint with $1\leq i\leq d$,  
  then $H'=\left. H\right|_{A'\times\ldots\times A'}$ and
  $H''=\left. H\right|_{A''\times\ldots\times A''}$ 
  are independent (and so, $L(H')$ and $L(H'')$ are also
  independent).
\end{enumerate}
\end{definition}
The reader may easily verify that the following distributions
  (on which we will focus attention) 
  give rise to random $d$-partite hyper-graph models:
\begin{itemize}
\item $\Sigma(K_{n_1,\ldots,n_d},k)$ (the \emph{random $d$-word model}) --- 
  the distribution over the set of hyper-subgraphs obtained
  from $K_{n_1,\ldots,n_d}$ when each element in the vertex set
  of $K_{n_1,\ldots,n_d}$ is uniformly and independently 
  randomly assigned one of $k$ letters and where edges,
  for which not all of its nodes end up being assigned the same 
  letter, are discarded.

\item $\calG(K_{n_1,\ldots,n_d},p)$ (the 
  \emph{$d$-dimensional binomial random hyper-graph model}) --- 
  the distribution over the set of 
  hyper-subgraphs $H$ of $K_{n_{1},\ldots,n_d}$ where 
  the events $\set{H}{e\in E(H)}$ for $e\in E(K_{n_1,\ldots,n_d})$
  have probability $p$ and are mutually independent.
\end{itemize}
The model $\Sigma(K^{(d)}_{n},k)$ is referred to as 
  the random word model because it arises when one considers
  the letters of $d$ words $\omega_1,\ldots,\omega_d$
  of length $n_1,\ldots,n_d$, respectively.
The letters in each word are chosen uniformly and independently 
  from a finite alphabet of size $k$.
Then, each word is identified with
  a color class of a hyper-subgraph~$H$ of $K_{n_1,\ldots,n_d}$
  whose hyper-edges are the $(v_1,\ldots,v_d)\in V(H)$
  for which $v_1,\ldots,v_d$ have been assigned the same letter.
It is easy to see that the longest common subsequence of 
  $\omega_1,\ldots,\omega_d$ equals $\ell$
  if and only if $L(H)=\ell$.
The random word model thus encompasses the Longest Common Subsequence
  problem discussed above.
Similarly, the attentive reader probably already noticed that 
  the binomial random graph model also encompasses the 
  already discussed
  point lattice process considered by Sepp\"al\"ainen~\cite{seppalainen97}.

Inspired in the work of Baik and Rains~\cite{br99} cited above,
  where symmetric variants of the Longest Increasing Subsequence problem
  were considered, we will also study the 
  following two symmetric variants of the 
  binomial random graph model:
\begin{itemize}
\item $\calS(K_{n,n},p)$ 
  (the \emph{symmetric binomial random graph model}) --- 
  the distribution over the set of subgraphs $H$ of $K_{n,n}$ where 
  the events $\set{H}{(i,j),(j,i)\in E(H)}$
  for $1\leq i<j\leq n$,
  have probability $p$ and are mutually independent.

\item $\calA(K_{2n,2n},p)$ 
  (the \emph{anti-symmetric binomial random graph model}) --- 
  the distribution over the set of 
  subgraphs $H$ of $K_{2n,2n}$ where 
  the events $\set{H}{(i,j),(2n-i+1,2n-j+1)\in E(H)}$ for
  $1\leq i,j\leq 2n$
  have probability $p$ and are mutually independent.
\end{itemize}
Note that $(\calS(K_{n,n},p))_{n\in\NN}$ is not a random model
  according to Definition~\ref{def:random}, 
  but it is a weak random model according to Definition~\ref{def:weak-random}.
On the other hand, $(\calA(K_{2n,2n},p))_{n\in\NN}$ is not even a 
  weak random model.

Henceforth, given a random bipartite graph model $\calD=(\calD(\cdot))$,
  any value that is constant across the distributions $\calD(\cdot)$
  will be called \emph{internal parameter} of the model --- e.g.~$1/p$ and 
  $k$ in $\calG(\cdot,p)$ and $\Sigma(\cdot,k)$, respectively. 

\begin{figure}
\begin{center}
\ifpdf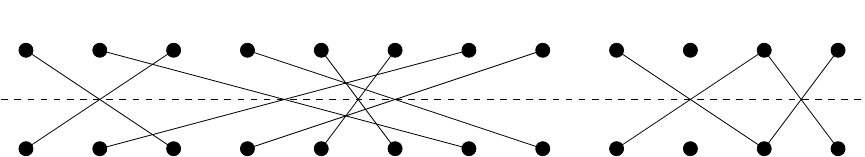\else\input{symmetric.pstex_t}\fi
\end{center}
\caption{A bipartite graph in the support of 
  $\calS(K_{12,12},p)$.}\label{fig:symmetric}
\end{figure}

\begin{figure}
\begin{center}
\ifpdf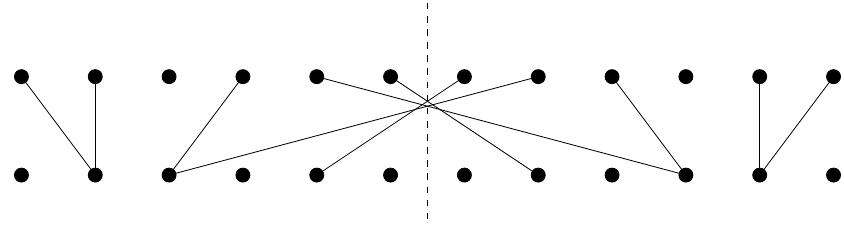\else\input{anti-symmetric.pstex_t}\fi
\end{center}
\caption{A bipartite graph in the support of 
  $\calA(K_{12,12},p)$.}\label{fig:anti-symmetric}
\end{figure}
The main purpose of this work is to establish a general result, 
  referred to as Main Theorem,
  with a minimal set of easily verifiable hypothesis, that 
  characterizes the limit behavior, when properly normalized, of 
  $\expec{}{L(\calD(K_{n}^{(d)},p))}$ when $d$ is fixed and
  both $n$ and the internal parameter $t$ go to infinity.
We also show several applications of our Main Theorem.
Specifically, we characterize aspects of the limiting behavior 
  for the four previously introduced random hyper-graphs models.
In the following section we formally state our Main Theorem and 
  the results of its application.

\subsection{Main contributions}
A straightforward application of Talagrand's inequality (as stated 
  in~\cite[Theorem~2.29]{jlr00}) yields that both
  $L(\Sigma(K_{n,n},k))$ and $L(\calG(K_{n,n},p))$ are
  concentrated around any one of their (potentially not unique) medians.
As we shall see, the same is true for 
    $L(\Sigma(K^{(d)}_{n},k))$ and $L(\calG(K^{(d)}_{n},p))$.
Somewhat equivalent statements hold
  for the the symmetric and anti-symmetric binomial random graph models.
The following general notion will encompass the concentration 
  type requirement 
  the random hyper-graph models will need to satisfy in order 
  for our Main Theorem to be applicable.
\begin{definition}
Let $\calF$ be a distribution over bipartite hyper-graphs
  and $\mediannopar$ be a median of $L(\calF)$. 
We say that $\calF$ has \emph{concentration constant $h$} 
  if for all $s\geq 0$,
\begin{eqnarray*}
\prob{}{L(\calF) \leq (1-s)\mediannopar} & \leq & 
  2\exp\left(-hs^2\mediannopar\right)\,, \\
\prob{}{L(\calF) \geq (1+s)\mediannopar} & \leq & 
  2\exp\left(-h\frac{s^2}{(1+s)}\mediannopar\right)\,.
\end{eqnarray*}
We say that the random bipartite hyper-graph model 
  $\calD=(\calD(\cdot))$ has concentration constant $h$ if each
  $\calD(\cdot)$ has concentration constant $h$.
\end{definition}

Note that if one can estimate a median of 
  $L(\calF)$ for some distribution $\calF$,
  show that the median and mean are close, and establish
  that $\calF$ has a concentration constant, then one can derive 
  a concentration (around its mean) result for $L(\calF)$.
Unfortunately, it is not in general easy to estimate a median 
  of $L(\calD(K_{n_1,\ldots,n_d}))$ for the distributions 
  $\calD(K_{n_1,\ldots,n_d})$ we consider.
However, we will be able to approximate them 
  under some assumptions on $n_1,\ldots,n_d$.
In particular, we will show that there is a median that 
  is proportional to the geometric mean of $n_1,\ldots,n_d$.
The following definition captures the aforementioned assumptions
  we will need, and the sort of approximation guarantee that we 
  will be able to establish.
\begin{definition}\label{def:approximate-median}
Led $\calD=(\calD(K_{n_1,\ldots,n_d}))$ be a random $d$-partite 
  hyper-graph model with internal parameter~$t$.
Fix $n_1,\ldots,n_d$ and let 
  $N=\left(\prod_{i=1}^{d}n_i\right)^{1/d}$ and 
  $S=\sum_{i=1}^{d}n_i$ denote the geometric mean
  and sum of $n_1,\ldots,n_d$, respectively.
We say that $\calD$ admits a \emph{$(c,\lambda,\theta)$-approximate median}
  (or simply a \emph{$(c,\lambda,\theta)$-median}) if for all $\delta>0$
  there are sufficiently large 
  constants $a(\delta)$, $b(\delta)$, and $t'(\delta)$, such that
  for all $t\geq t'$, for which 
\begin{itemize}
\item \textbf{Size lower bound condition:} 
  $N \geq at^{\lambda}$,

\item \textbf{Size upper bound condition:} 
  $Sb\leq t^{\theta}$,
\end{itemize}
it holds that
\[
(1-\delta)\frac{cN}{t^{\lambda}} \ \leq \
\median{}{L(\calD(K_{n_1,\ldots,n_d}))}
  \ \leq \ 
  (1+\delta)\frac{cN}{t^{\lambda}}\,.
\]  
\end{definition}
In other words, if $\calD=(\calD(K_{n_1,\ldots,n_d}))$ is a  a random 
  $d$-partite 
  hyper-graph model with internal parameter~$t$
  that admits a $(c,\lambda,\theta)$-median and the geometric mean
  (respectively sum) of $n_1,\ldots,n_d$ is $N$ (respectively~$S$)
  are such that $N=\Omega(t^{\lambda})$ (respectively $S=O(t^{\theta})$),
  then for sufficiently large $t$, every median of $L(\calD(n_1,\ldots,n_d))$
  will be close to $cNt^{-\lambda}$.
Although the above defined approximate median notion might at first glance
  sound artificial, we will see that it is possible to 
  obtain such type of approximations for the random
  hyper-graph models we are interested on. 

Returning to our discussion, the relevance of 
  the notion of approximate median
  is, when the random hyper-graph model admits a concentration constant, 
  that it allows us to derive concentration bounds around an  
  approximation of the median which in turn will be closed to the mean.
Endowed with such estimates of the mean, we can easily derive the 
  thought after limiting behavior of such expected values.
This in essence, is the crux of our approach to attacking all
  variants of the Largest Non-crossing Matching problem.

Unfortunately, the approximation of $\median{}{L(\calD(n_1,\ldots,n_d))}$
  guaranteed by the existence of a $(c,\lambda,\theta)$-median,
  as in Definition~\ref{def:approximate-median}, holds 
  for the rather restrictive condition 
  $b\sum_{i=1}^{d} n_{i} \leq t^{\theta}$.
However, the monotonicity and block independence properties 
  of random hyper-graph models allow us to relax
  the restriction and still obtain essentially the same conclusion.
More precisely, it will be possible to obtain the same guarantee,
  but requiring only that the sum of the $n_i$'s is not
  too large in comparison with the geometric mean of the $n_i$'s.
Moreover, and of crucial importance, under the same conditions one
  can show that the median and mean of $L(\calD(n_1,\ldots,n_d))$
  are close to each other. 
The following result, which is the main result of this work,
  precisely states the claims made in the preceding informal discussion.
\begin{theorem}{[Main Theorem]}
Let $\calD=(\calD(K_{n_1,\ldots,n_d}))$ be a random hyper-graph model with 
  internal parameter $t$ and concentration constant $h$ which admits 
  a $(c,\lambda,\theta)$-median.
Fix $n_1,\ldots,n_d$ and let $N$ and $S$ denote the geometric mean
  and sum of $n_1,\ldots,n_d$, respectively.
Let $0\leq \eta\leq \min\setof{\lambda/(d-1),\theta-\lambda}$
  and $g=O(t^{\eta})$.

For all $\epsilon>0$ there exists $t_0$ and $A$ sufficiently large 
  such that if $t\geq t_0$ is such that
  $N\geq t^{\lambda}A$ (size constraint) and
  $S\leq g(t)N$ (balance condition), then 
\begin{eqnarray}\label{eqn:maintheo-expec}
(1-\epsilon)\frac{cN}{t^{\lambda}} 
  \ \leq &
  \expec{}{L(\calD(K_{n_1,\ldots,n_d}))} 
  & \leq \
  (1+\epsilon)\frac{cN}{t^{\lambda}}\,,
\end{eqnarray}
and the following hold:
\begin{itemize}
\item If $\mediannopar$ is a median of 
  $\median{}{L(\calD(K_{n_1,\ldots,n_d}))}$, then
\begin{eqnarray}\label{eqn:maintheo-median}
(1-\epsilon)\frac{cN}{t^{\lambda}} 
  \ \leq &
  \mediannopar
  & \leq \
  (1+\epsilon)\frac{cN}{t^{\lambda}}\,.
\end{eqnarray}

\item There is a constant $K>0$ such that 
\begin{eqnarray}
\prob{}{L(\calD(K_{n_1,\ldots,n_d}))\leq (1-\epsilon)\frac{cN}{t^{\lambda}}}
  & \leq & \exp\left(-Kh\epsilon^2\frac{cN}{t^{\lambda}}\right)\,, 
  \label{eqn:maintheo-lower}
\\
\prob{}{L(\calD(K_{n_1,\ldots,n_d}))\geq (1+\epsilon)\frac{cN}{t^{\lambda}}}
  & \leq & 
  \exp\left(-Kh\frac{\epsilon^2}{1+\epsilon}\frac{cN}{t^{\lambda}}\right)\,.
  \label{eqn:maintheo-upper}
\end{eqnarray}
\end{itemize}
Moreover, if $n_{1}=\ldots=n_{d}=n$ and 
  $\calD=(\calD(K_{n}^{(d)}))$ is just a weak random hyper-graph model,
  then the the lower bounds
  in~\eqref{eqn:maintheo-expec} 
  and~\eqref{eqn:maintheo-median}, 
  and inequality~\eqref{eqn:maintheo-lower}, still hold.  
\end{theorem}

As a consequence of the previously stated Main Theorem, with
  some additional work, we can derive 
  several results concerning the asymptotic behavior of 
  the expected length of a largest non-crossing matching for all
  of the random models introduced above.
Our first two applications of the Main Theorem concern
  the random binomial hyper-graph model $(\calG(K_{n}^{(d)},p))_{n\in\NN}$
  and the random word model $(\Sigma(K_{n}^{(d)},k))_{n\in\NN}$.
The asymptotic behavior of the length of a largest non-crossing 
  hyper-matching for both of these models is (interestingly!) 
  related to a constant $c_{d}$ that arises in the work of
  Bollob\'as and Winkler~\cite{bw88} concerning the height
  of a largest chain among random points  
  independently chosen in the $d$-dimensional unit cube $[0,1]^{d}$.
Specifically, for the random binomial hyper-graph model, we show:
\begin{theorem}\label{th:main-binomial-model}
For $0<p<1$, there exists a constant $\delta_{p}$ such that 
\[
\lim_{n\to \infty} \frac{1}{n}\expec{}{L(\calG(K^{(d)}_{n},p))}
  \ \ = \ \ 
  \inf_{n\in\NN} \frac{1}{n}\expec{}{L(\calG(K^{(d)}_{n},p))}
  \ \ = \ \ 
  \delta_{p}\,,
\]
and $\delta_{p}/\sqrt[d]{p}\to c_{d}$ when $p\to 0$.
\end{theorem}
For the case where the underlying model is the one that 
  arises when interested in the length of a longest common subsequence
  of $d$ randomly chosen words over a finite alphabet, i.e.~the 
  random $d$-word model, we establish:
\begin{theorem}\label{th:main-word-model}
For $k\in\NN$, there exists a constant $\gamma_{k}$ such that 
\[
\lim_{n\to \infty} \frac{1}{n}\expec{}{L(\Sigma(K^{(d)}_{n},k))}
  \ \ = \ \ 
  \inf_{n\in\NN} \frac{1}{n}\expec{}{L(\Sigma(K^{(d)}_{n},k))}
  \ \ = \ \ 
  \gamma_{k}\,,
\]
and $k^{1-1/d}\gamma_{k}\to c_{d}$ when $k\to\infty$.
\end{theorem}
The $d=2$ case of Theorems~\ref{th:main-binomial-model} 
  and~\ref{th:main-word-model} were already established
  by Kiwi, Loebl, Matou\v{s}ek~\ref{klm05}.
This work generalizes and strengthens the arguments
  developed in~\ref{klm05}, as well as elicits new connections
  with other previously studied problems (most notably 
  in~\cite{bw88}).

\medskip
Finally, we consider the symmetric versions of random
  graph models introduced above and 
  show how the Main Theorem, plus some additional 
  observations, allows one to characterize some aspects of 
  the asymptotic behavior of the length of a longest non-crossing
  matching.
Specifically, we prove the following two results.
\begin{theorem}\label{th:main-random-symm}
For $0<p<1$, there exists a constant $\sigma_{p}$ such that 
\[
\lim_{n\to \infty} \frac{1}{n}\expec{}{L(\calS(K_{n,n},p))}
  \ \ = \ \ 
  \inf_{n\in\NN} \frac{1}{n}\expec{}{L(\calS(K_{n,n},p))}
  \ \ = \ \ 
  \sigma_{p}\,,
\]
and $\sigma_{p}/\sqrt{p}\to 2$ when $p\to 0$.
\end{theorem}

\begin{theorem}\label{th:main-random-anti-symm}
For $0<p<1$, there exists a constant $\alpha_{p}$ such that 
\[
\lim_{n\to \infty} \frac{1}{2n}\expec{}{L(\calA(K_{2n,2n},p))}
  \ \ = \ \ 
  \inf_{n\in\NN} \frac{1}{2n}\expec{}{L(\calA(K_{2n,2n},p))}
  \ \ = \ \ 
  \alpha_{p}\,,
\]
and $\alpha_{p}/\sqrt{p}\to 2$ when $p\to 0$.
\end{theorem}

\subsection{Preliminaries}
For future reference we determine below concentration constants
  for the binomial and word models.
\begin{proposition}\label{prop:model-concentration-constants}
The $d$-dimensional binomial random hyper-graph model admits a 
  concentration constant of $1/4$.
The random $d$-word model admits a concentration constant of 
  $1/(4d)$.
\end{proposition}
\begin{proof}
Let $H$ be chosen according to $\calG(K_{n_1,\ldots,n_d},p)$. 
Since $L(H)$ depends exclusively on whether or not an edge appears
  in $H$ (and by independence among these events), it follows 
  that $L(H)$ is $1$-Lipschitz, i.e.~$|L(H)-L(H\triangle\setof{e})|\leq 1$.
Moreover, if $L(H)\geq r$, then there is a set of $r$ edges that are
  a witness for the fact that $L(H)\geq r$, for every $H$ containing
  such a set of $r$ edges.
A direct application of Talagrand's inequality (as stated 
  in~\cite[Theorem~2.29]{jlr00}) proves the claim about the 
  concentration constant for the $d$-dimensional binomial random
  hyper-graph model.
The case of the random $d$-word model is similar and left to the 
  reader to verify.
\end{proof}

\subsection{Organization:}
For the sake of clarity of exposition 
  and given that the arguments employed are different,
  we prove in separate sections 
  the lower and upper bounds (as well as lower and upper tail bounds)
  of the Main Theorem's statement.
Specifically, in Section~\ref{sec:lower}, we establish all the lower bounds 
  and lower tail bounds claimed in the Main Theorem.
In Section~\ref{sec:upper}, we prove the upper bounds 
  and upper tail bounds stated in the Main Theorem, 
  thence completing its proof.
Finally, in Section~\ref{sec:appl}, we apply the Main Theorem 
  to four distinct scenarios.
Specifically, we consider the cases where the underlying 
  random model is the binomial random hyper-graph
  model, the random word model, the symmetric binomial
  random graph model, and the anti-symmetric binomial random
  graph model.

\section{Lower bounds}\label{sec:lower}
In this section we will establish the lower bounds claimed in 
  the statement of the Main Theorem, i.e.~the lower bounds
  in~\eqref{eqn:maintheo-expec} 
  and~\eqref{eqn:maintheo-median}, 
  and inequality~\eqref{eqn:maintheo-lower}.

Let $\calD$, $c$, $\lambda$, $\theta$, $\eta$, and $\epsilon$ be 
  as in the statement of the Main Theorem.
Let $\delta >0$ be sufficiently small so
\begin{equation}\label{eqn:def-delta}\tag{Definition of $\delta$}
(1-\delta)^{2}(1-2\delta) \ \geq \ 1-\frac{\epsilon}{2}\,,
\end{equation} 
and let $a=a(\delta)$, $b=b(\delta)$ and $t'=t'(\delta)$ 
  as guaranteed by the definition of $(c,\lambda,\theta)$-median.

Since $g=O(t^{\eta})$, there are constants $C_g>1$ 
  and $t_g\geq 0$ such that $g(t)\leq C_gt^{\eta}$ for all $t\geq t_g$.
Choose $A$ large enough so 
\begin{eqnarray}\label{eqn:choiceA-lower}
A & \geq & \max\setof{
  \frac{2a}{1-\delta}, 
  \frac{2}{1-(1-\delta)^{d}}C_{g}^{d-1}t_{g}^{\eta(d-1)-\lambda},
  \frac{2}{h\delta^2c}\ln(2/\delta),
  \frac{16\ln(2)}{hc\epsilon^{2}}}\,.
\end{eqnarray}
Choose $t_0>\max\setof{t_{g},t'(\delta)}$ sufficiently large so that for all 
  $t\geq t_0$,
\begin{eqnarray}\label{eqn:sizes1}
g(t)\ \leq \ C_{g}t^{\eta} 
  & \text{ and } &
  C_{g}bAt^{\eta} \ \leq \ t^{\theta-\lambda}\,.
\end{eqnarray}
Now, assume $t>t_0$ and that the geometric mean $N$ and sum $S$
  of $n_1,\ldots,n_d$ satisfy the size and balance conditions.
Thus, the size constraint and balance condition guarantee that 
\begin{eqnarray}\label{eqn:sizes2}
N\ \geq \ At^{\lambda} 
  & \text{ and } &
  S \ \leq \ g(t)N \ \leq \ C_{g}Nt^{\eta}\,.
\end{eqnarray}
Finally, assume $H$ is chosen according to $\calD(K_{n_1,\ldots,n_d})$.

If the $n_i$'s satisfy the size conditions of 
  the definition of a $(c,\lambda,\theta)$-median
  and since the model admits a concentration constant, then
  we would have a concentration bound around 
  $cNt^{\lambda}$ for $L(H)$.
Unfortunately, when some of the $n_i$'s are large, then $S$ will
  be large, and the size upper bound condition need not be 
  satisfied, leaving us without the desired concentration bound.
To overcome this situation, we break apart $H$ into hyper-subgraphs
  $H_{1},\ldots, H_{q}$
  of roughly the same size which we will refer to as \emph{blocks}.
The blocks will be vertex disjoint, the proportion between the sizes
  of the color classes in each $H_i$ will be roughly the same
  than the one in $H$.
However, the crucial new aspect is that the size upper bound
  condition will be satisfied in each block $H_{i}$ allowing
  us to derive a concentration bound for $L(H_i)$.
This will later allow us to obtain a concentration bound for
  $L(H)$, details follow.

Let $q=\lceil N/(At^{\lambda})\rceil$. 
For each $1\leq j\leq d$, let $n_j'=\lfloor n_j/q\rfloor$.
Henceforth, let $N'$ and $S'$ denote the geometric mean and the sum 
  of the $n'_j$'s.
Denote the $j$-th color class of $H$ by~$A_j$.
Recall that $A_j$ is totally ordered. 
Let $A_{j,1}$ be the first $n'_j$ elements of $A_{j}$, 
    $A_{j,2}$ be the following $n'_j$ elements of $A_{j}$, 
    so on and so forth up to defining $A_{j,q}$.
Clearly, the $A_{j,i}$'s are disjoint, but do not necessarily cover
  all of $A_{j}$.
Now, for $1\leq i\leq q$, define $H_{i}$ as the hyper-subgraph induced
  by $H$ in $A_{1,i}\times\ldots\times A_{d,i}$ 
  (for an illustration, see Figure~\ref{fig:block-partition}).
Observe that the proportion between the sizes of the color classes
  of $H_i$ is roughly the same as the one among the 
  color classes of $H$.
\begin{figure}
\begin{center}
\ifpdf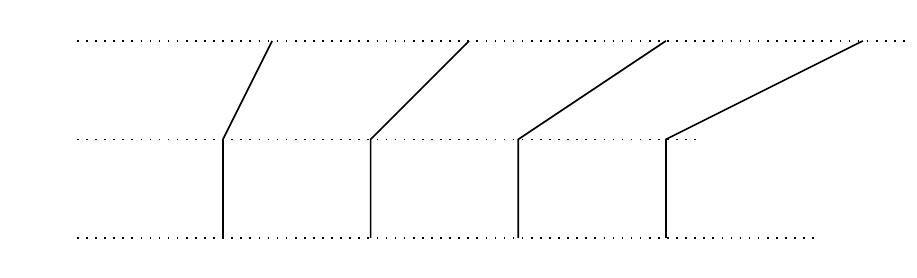\else\input{hyper-blocks.pstex_t}\fi
\end{center}
\caption{Illustration, for $d=3$ and $q=4$, of the construction of blocks
  $H_1,\ldots, H_q$.}\label{fig:block-partition}
\end{figure}

Note that by monotonicity, the distribution of 
  $H_{i}$ is $\calD(K_{n'_1,\ldots,n'_d})$.
Moreover, since the $H_{i}$'s are disjoint, by block independence, 
  their distributions.
It follows that $L(H_1),\ldots,L(H_q)$ are independent random variables.
A crucial, although trivial, observation is that 
\begin{eqnarray}\label{eqn:Hsubad}
L(H) & \geq & \sum_{i=1}^{q} L(H_i)\,.
\end{eqnarray}  
On the other hand, by definition of $q$ and the size constraint condition,
\begin{equation}\label{eqn:estimate-q}\tag{Estimate of $q$}
\frac{N}{At^{\lambda}} 
  \ \leq \ q 
  \ \leq \ \frac{N}{At^{\lambda}} + 1
  \ \leq \ \frac{2N}{At^{\lambda}}\,.
\end{equation}

In order to estimate the geometric mean $N'$ of 
  $n'_{1},\ldots,n'_{d}$, the following result will be useful.
\begin{lemma}
If $x_1,\ldots,x_d$ are positive real numbers, then
\begin{eqnarray*}
\prod_{j=1}^{d}(x_j-1) & \geq & \prod_{j=1}^{d} x_{j}
  - \left(\sum_{j=1}^{d} x_{j}\right)^{d-1}\,.
\end{eqnarray*}
\end{lemma}
\begin{proof}
By induction on $d$.
\end{proof}
It follows, by the preceding lemma, the estimate of $q$, and 
  the balance condition, that
\begin{eqnarray*}
\lefteqn{\frac{N}{q} 
  \ \ = \ \ \prod_{j=1}^{d}\left(\frac{n_{j}}{q}\right)^{1/d}
  \ \ \geq \ \ N' 
  \ \ \geq \ \ 
  \left(\prod_{j=1}^{d}\left(\frac{n_j}{q}-1\right)\right)^{1/d} } \\
  & & \ \geq \ \ 
  \left(\prod_{j=1}^{d} \frac{n_j}{q} - 
  \left(\sum_{j=1}^{d}\frac{n_j}{q}\right)^{d-1}\right)^{1/d}
  \ \ = \ \ 
  \frac{N}{q}\left(1-q\frac{S^{d-1}}{N^{d}}\right)^{1/d}\,.
\end{eqnarray*}
By~\eqref{eqn:sizes2}, our estimate of $q$, and since
  $\eta(d-1)<\lambda$, 
\[
q\frac{S^{d-1}}{N^{d}} 
  \ \leq \ \frac{2}{A}C_{g}^{d-1}t^{\eta(d-1)-\lambda}\,.
\]
Given the way we have chosen $A$, we have that 
  $(1-qS^{d-1}/N^{d})^{1/d}\geq 1-\delta$ and thus
\begin{equation}\label{eqn:estimate-Nprime}\tag{Estimate of $N'$}
\frac{N}{q} \ \geq \ N' \ \geq \ \frac{N}{q}(1-\delta)\,.
\end{equation}
Based on the preceding estimate of $N'$ and the estimate for 
  $q$ we will now show that $n'_1,\ldots,n'_d$ satisfy the 
  size conditions required by the definition of 
  $(c,\lambda,\theta)$-median.
Indeed, by our estimate of $N'$ and $q$, 
  and~\eqref{eqn:choiceA-lower} 
\begin{equation*}
N' \ \geq \
  \frac{N}{q}(1-\delta)
  \ \geq \ 
  \frac{1}{2}At^{\lambda}(1-\delta)
  \ \geq \ 
  at^{\lambda}\,.
\end{equation*}
Moreover, by definition of $S'$, our estimate of $q$,
  \eqref{eqn:sizes2}, and~\eqref{eqn:sizes1}, 
\begin{equation*}
S'b \ \leq \
  \frac{Sb}{q} 
  \ \leq \ 
  \frac{SbAt^{\lambda}}{N}
  \ \leq \
  C_{g}bAt^{\lambda+\eta}
  \ \leq \
  t^{\theta}\,.
\end{equation*} 
Now, let $H'$ be chosen according to 
  $\calD(K_{n'_1,\ldots,n'_d})$ and let $\mediannopar'$ be
  a median of $L(\calD(K_{n'_1,\ldots,n'_d}))$.
By definition
  of $(c,\lambda,\theta)$-median, we get that
  $cN't^{-\lambda}(1-\delta)\leq\mediannopar'\leq cN't^{-\lambda}(1+\delta)$.
Moreover, by definition of constant of concentration and approximate median,
  applying Markov's inequality yields,
\begin{eqnarray*}
\expec{}{L(H')} 
   & \geq &
   (1-2\delta)\frac{cN'}{t^{\lambda}}
   \prob{}{L(H')\geq (1-2\delta)\frac{cN'}{t^{\lambda}}} \\
   & \geq &
   (1-2\delta)\frac{cN'}{t^{\lambda}}
   \prob{}{L(H')\geq \left(1-\frac{\delta}{1-\delta}\right)\mediannopar'} \\
   & \geq &
   (1-2\delta)\frac{cN'}{t^{\lambda}}
   \left(1-2\exp\left(-h\frac{\delta^{2}}{(1-\delta)^{2}}\mediannopar'\right)\right) \\
   & \geq &
   (1-2\delta)\frac{cN'}{t^{\lambda}}
   \left(1-2\exp\left(-h\frac{\delta^{2}}{1-\delta}\frac{cN'}{t^{\lambda}}\right)\right)\,.  
\end{eqnarray*}
As observed above, $N'\geq At^{\lambda}(1-\delta)/2$, so by choice of 
  $A$, we get that $\expec{}{L(H')}\geq (1-2\delta)(1-\delta)cN't^{-\lambda}$.
Hence, given that $L(H)\geq \sum_{i=1}^{q}L(H_{i})$, 
  the estimate of $N'$, the definition of $\delta$, and elementary
  algebra,
\[
\expec{}{L(H)} 
  \ \geq \
  \sum_{i=1}^{q}\expec{}{L(H_{i})}
  \ \geq \
  (1-2\delta)(1-\delta)q\frac{cN'}{t^{\lambda}}
  \ \geq \
  (1-2\delta)(1-\delta)^{2}\frac{cN}{t^{\lambda}}
  \ \geq \
  (1-\epsilon/2)\frac{cN}{t^{\lambda}}\,.
\]
We have thus established the lower bound claimed 
  in~\eqref{eqn:maintheo-expec}.

Now, we proceed to show~\eqref{eqn:maintheo-lower}.
Note that
\begin{eqnarray}\label{eqn:maintheo-lower-aux} 
\prob{}{L(H)\leq (1-\epsilon)\frac{cN}{t^{\lambda}}}
  & \leq & \sum_{\substack{(s_1,\ldots,s_q)\in\NN^{q} \\ s_1+\ldots+s_q\leq (1-\epsilon)cNt^{-\lambda}}}
  \prob{}{L(H_i) = s_i, i=1,\ldots,q}\,.
\end{eqnarray}
Let $\calT$ be the set of indices of the summation in the preceding
  displayed equation.
Also, for $T=(s_1,\ldots,s_q)$ belonging to $\calT$  let $P_T$ denote 
  $\prob{}{L(H_i)=s_i,i=1,\ldots,q}$.
We will show that $P_{T}$ is exponentially small with respect to 
  $cNt^{-\lambda}$.
Recalling that the $L(H_{i})$'s are independent and  
  distributed as $L(H')$ when $H'$ is chosen 
  according to $\calD(K_{n'_1,\ldots,n'_d})$,
\[
P_{T} \ = \ \prod_{i=1}^{q} \prob{}{L(H_i)=s_i}
  \ \leq \ \left(\prob{}{L(H')\leq s_i}\right)^{q}\,.
\]
Again, by the way in which $H'$ is chosen, the definition of 
  $\mediannopar'$, and the definition of $(c,\lambda,\theta)$-median, 
  for all $i$ such that $s_i\leq (1-\delta)cN't^{-\lambda}\leq
    \mediannopar' \leq (1+\delta)cN't^{-\lambda}\leq 2cN't^{-\lambda}$,
  it holds that
\begin{eqnarray*}
\lefteqn{\prob{}{L(H')\leq s_i} \ \ = \ \ 
  \prob{}{L(H') \leq \left(1-\frac{\mediannopar'-s_i}{\mediannopar'}
     \right)\mediannopar'}} \\
  & \leq & 2\exp\left(-h\frac{(\mediannopar'-s_i)^{2}}{\mediannopar'}\right)
  \ \ \leq \ \ 2\exp\left(
    -\frac{ht^{\lambda}}{2cN'}((1-\delta)cN't^{-\lambda}-s_i)^2\right)\,.
\end{eqnarray*}
Hence, for all $1\leq i \leq q$,
\begin{eqnarray*}
\prob{}{L(H')\leq s_i} 
  & \leq & 2\exp\left(
    -\frac{ht^{\lambda}}{2cN'}
      \max\setof{0,((1-\delta)cN't^{-\lambda}-s_i)^2}\right)\,,
\end{eqnarray*}
and then
\[
-\ln P_{T} \ \geq \ 
  -\sum_{i=1}^{q}\ln\prob{}{L(H')\leq s_i}
  \ \geq \ -q\ln(2) + \frac{ht^{\lambda}}{2cN'}
    \sum_{i=1}^{q}\left(\max\setof{0,(1-\delta)cN't^{-\lambda}-s_i}\right)^{2}\,.
\]
By Cauchy-Schwartz's inequality, our estimate of $N'$,
  the fact that $s_1+\ldots+s_q\leq (1-\epsilon)cNt^{-\lambda}$, 
  and since by definition of $\delta$ we know that
  $(1-\delta)^{2}\geq 1-\epsilon/2$,
\begin{eqnarray*}
\lefteqn{\sqrt{q\sum_{i=1}^{q}\left(\max\setof{0,(1-\delta)cN't^{-\lambda}-s_i}\right)^2}
  \ \ \geq \ \
  \sum_{i=1}^{q}\max\setof{0,(1-\delta)cN't^{-\lambda}-s_i}} \\
  & & \quad \geq \ \ 
  (1-\delta)cN'qt^{-\lambda}-\sum_{i=1}^{q}s_i 
  \ \ \geq \ \
  (1-\delta)^{2}cNt^{-\lambda}-(1-\epsilon)cNt^{-\lambda} 
  \ \ \geq \ \
  \frac{cN\epsilon}{2t^{\lambda}}\,.
\end{eqnarray*}
Combining the last two displayed inequalities
  and recalling our estimate of $N'$,
  we get 
\[
-\ln P_{T} 
  \ \geq \
  -q\ln(2)+\frac{ht^{\lambda}}{2cN'q}\cdot\frac{c^2N^2\epsilon^2}{4t^{2\lambda}}
  \ \geq \
  -q\ln(2)+\frac{hcN\epsilon^2}{8t^{\lambda}}\,.
\]
By~\eqref{eqn:maintheo-lower-aux} and
  using the standard estimate ${a\choose b}\leq (ea/b)^{b}$,
  we have
\begin{eqnarray*}
\lefteqn{\prob{}{L(H)\leq (1-\epsilon)\frac{cN}{t^{\lambda}}}
  \ \ \leq \ \ 
  \sum_{T\in\calT} P_{T}
  \ \ \leq \ \
  \left|\calT\right|\cdot\max_{T\in\calT} P_{T} } \\
  & \leq &
  {\lfloor (1-\epsilon)cNt^{-\lambda}\rfloor + q\choose q}\cdot
     \max_{T\in\calT} P_{T} 
  \ \ \leq \ \
  \exp\left(q\ln\left(2e[1+(1-\epsilon)cNt^{-\lambda}/q]\right)-\frac{hcN\epsilon^{2}}{8t^{\lambda}}\right)\,.
\end{eqnarray*}
Now, by $q$'s estimate we know that $N\leq qAt^{\lambda}\leq 2N$.
Thus, if we require that $A$ is large enough so that 
  $\ln(2e[1+(1-\epsilon)cA])
  \leq Ahc\epsilon^{2}/32$, we get that
\[
\prob{}{L(H)\leq (1-\epsilon)\frac{cN}{t^{\lambda}}}
  \ \leq \
  \exp\left(\frac{2N}{At^{\lambda}}
  \ln(2e[1+(1-\epsilon)cA])-\frac{hcN\epsilon^2}{8t^{\lambda}}\right)
  \ \leq \ \exp\left(-\frac{h\epsilon^2cN}{16t^{\lambda}}\right)\,.
\]
This proves the lower bound claimed in~\eqref{eqn:maintheo-lower}.

What remains is to show the lower bound in~\eqref{eqn:maintheo-median}.
By $q$'s estimate we have $N\geq At^{\lambda}$ which
  together with our choice of $A$ (see~\eqref{eqn:choiceA-lower}), 
  imply that
\[
\exp\left(-\frac{h\epsilon^2cN}{16t^{\lambda}}\right)
  \ \leq \ 
  \exp\left(-\frac{h\epsilon^2c}{16}A\right)
  \ \leq \
  \frac{1}{2}\,.
\]
Combining the last two displayed equations, it follows that 
  $\prob{}{L(H)\leq (1-\epsilon)cNt^{-\lambda}}\leq 1/2$, 
  implying that any median of $L(H)$ must be at least 
  $(1-\epsilon)cNt^{-\lambda}$.

\begin{remark}
The reader may check that all claims proved in this section still hold
  if instead of $\calD=(\calD(K_{n_1,\ldots,n_d}))$ we had worked 
  with a weak random hyper-graph model $\calD=(\calD(K^{(d)}_{n}))$.
Indeed, if this would have been the case, then 
  for $H$ chosen according to $\calD(K^{(d)}_{n})$, 
  the hyper-graphs $H_{1},\ldots,H_{q}$ obtained above from $H$ 
  would have all their color classes of equal size, and the 
  weak random hyper-graph model assumption is all that is
  all that is need to carry forth the arguments laid out in this
  section.
\end{remark}

\section{Upper bounds}\label{sec:upper}
In this section we will establish the upper bounds claimed in 
  the statement of the Main Theorem, i.e.~the upper bounds
  in~\eqref{eqn:maintheo-expec} and~\eqref{eqn:maintheo-median}, 
  and inequality~\eqref{eqn:maintheo-upper}.
The proof of the latter of these bounds, the 
  upper tail bound, is rather long.
For sake of clarity of exposition, we have divided 
  its proof in three parts. 
First, in Section~\ref{sec:upper-notation}, 
  we introduce some useful variables.
In Section~\ref{sec:upper-tail-small},
  we establish~\eqref{eqn:maintheo-upper} for 
  not to large values of the geometric mean $N$.
Then, in Section~\ref{sec:upper-tail-large},
  we consider the case where $N$ is large.
Finally, in Section~\ref{sec:upper-mean-median}, 
  we conclude the proof of the bounds claimed
  in the Main Theorem.

\subsection{Basic variable definitions}\label{sec:upper-notation}
For the rest of this section, 
  let $\calD$, $c$, $\lambda$, $\theta$, $\eta$, and $\epsilon$ be 
  as in the statement of the Main Theorem.
Define 
\begin{equation}\label{eqn:def-delta-upper}\tag{Definition of $\delta$}
\delta \ = \ \min\setof{1,\frac{\epsilon^{2}}{1+\epsilon},
  \frac{\epsilon}{6}}\,.
\end{equation} 
Let $a=a(\delta)$, $b=b(\delta)$ and $t'=t'(\delta)$ 
  as guaranteed by the definition of $(c,\lambda,\theta)$-median.
Choose $A$ so
\begin{eqnarray}\label{eqn:choiceA-upper}
A & = & \max\setof{
  \frac{a}{\delta}, 
  \frac{8\ln(2)}{h\delta c}}\,.
\end{eqnarray}
For technical reasons, it will be convenient to fix constants $\alpha$ 
  and $\beta$ such that
\begin{equation}
\lambda \ < \ \alpha \ < \ \beta \ < \ \theta-\eta\,.
\end{equation}
We shall also encounter two constants $K_1$ and $K_2$, depending solely
  on $d$.
Since $g=O(t^{\eta})$, there are constants $C_g>1$ 
  and $t_g\geq t'$ such that for all $t\geq t_g$ it holds that 
  $g(t)\leq C_gt^{\eta}$, and 
\begin{eqnarray}
\max\setof{9At^{\lambda},e} \ \ \leq \ \ 
  t^{\alpha} & \leq & t^{\beta} 
  \ \ \leq \ \ \frac{1}{bdC_{g}}t^{\theta-\eta}\,, 
  \label{eqn:constant} \\ 
\frac{2t^{\lambda}}{c}
  \ \ \leq \ \ 
  \frac{2\alpha K_1}{\delta c K_{2} h}t^{\lambda} 
  & \leq & \frac{t^{\alpha}}{\ln(t)}\,.
  \label{eqn:constantK} 
\end{eqnarray}
Consider now $t\geq t_{g}$ and the positive integers 
  $n_1,n_2,\ldots,n_d$ with geometric mean $N$, summing up to $S$,
  and satisfying both the size constraint condition ($N\geq At^{\lambda}$)
  and balance condition ($S\leq g(t)N$).
Furthermore, define $M=cNt^{-\lambda}$ and choose $H$ according to 
  $\calD(K_{n_1,\ldots,n_d})$.
In the following two sections, we separately consider the case 
  where $N$ is less than and at least $t^{\beta}$.

\subsection{Upper tail bound for not to large values of $N$}\label{sec:upper-tail-small}
Throughout this section, we assume $N<t^{\beta}$.
We will show that $H$ satisfies the size lower bound restriction 
  in the definition of $(c,\lambda,\theta)$-median.
The fact that $\calD$ admits a concentration constant $h$
  will allow us obtain a bound on the upper tail of $L(H)$.

Let $t\geq t_{g}$.
Since $N$ satisfies both the size constraint and balance condition,
  by~\eqref{eqn:choiceA-upper}, \eqref{eqn:constant}, and 
  the definition of $\delta$,
\begin{eqnarray*}
N \ \ \geq \ \ At^{\lambda} & \geq & \frac{at^{\lambda}}{\delta}
  \ \ \geq \ \ at^{\lambda}\,, \\
Sb \ \ \leq \ \ g(t)bN & \leq & C_{g}bt^{\eta+\beta} \ \leq \ 
  \frac{t^{\theta}}{d} \ \ \leq \ \ t^{\theta}\,.
\end{eqnarray*}
Thus, $n_1,\ldots,n_d$ satisfy both the size lower and upper bound
  conditions of the definition of $(c,\lambda,\theta)$-median.
Hence, if $H$ is chosen according to $\calD(K_{n_1,\ldots,n_d})$,
  then every median $\mediannopar$
  of $L(H)$ is $\delta M$ close to~$M$.
Simple algebra, the definitions of concentration constant 
  and $(c,\lambda,\theta)$-median,
  and given that by definition of $\delta$
  we know that $\delta<\epsilon$, we have
\begin{eqnarray*}
\lefteqn{\prob{}{L(H) \geq (1+\epsilon)M} 
  \ \ = \ \ \prob{}{L(H) \geq 
  \left(1+\frac{(1+\epsilon)M-\mediannopar}{\mediannopar}\right)\mediannopar}}\\
  & \leq & 2\exp\left(
    -h\frac{((1+\epsilon)M-\mediannopar)^{2}}{(1+\epsilon)M}
  \right) 
  \ \ \leq \ \
  \exp\left(\ln(2)-h\frac{(\epsilon - \delta)^{2}}{1+\epsilon}M
  \right)\,.
\end{eqnarray*}
By~\eqref{eqn:choiceA-upper}, 
  since $N\geq At^{\lambda}$,
  the fact that by definition of $\delta$ we know
  that $\delta \leq \epsilon^{2}/(1+\epsilon)$, 
  and recalling that $M=cNt^{-\lambda}$, 
\begin{eqnarray*}
\lefteqn{\prob{}{L(H) \geq (1+\epsilon)M} 
  \ \ \leq \ \
  \exp\left(
    \frac{Ah\delta c}{8} - \frac{h\epsilon^{2}}{4(1+\epsilon)}M
  \right)} \\
  & \leq & 
  \exp\left(
    \frac{h\delta Nc}{8t^{\lambda}} - \frac{h\epsilon^{2}}{4(1+\epsilon)}M
  \right) 
  \ \ \leq \ \
  \exp\left(-\frac{h\epsilon^{2}}{8(1+\epsilon)}M\right)\,. 
\end{eqnarray*}
We have thus established~\eqref{eqn:maintheo-upper} for $N<t^{\beta}$.

\subsection{Upper tail bound for large values of $N$}\label{sec:upper-tail-large}
We now consider the case where $N\geq t^{\beta}$.
The magnitude of $N$ is such that we can not directly apply the 
  definition of $(c,\lambda,\theta)$-median to a hyper-graph
  generated according to $\calD(K_{n_1,\ldots,n_d})$, and thus derive
  the sought after exponentially small tail bound.
We again resort to the block partitioning technique 
  introduced in the proof of the lower bound. 
However, both the block partitioning and the analysis are more 
  delicate and involved in the case of the upper bound.

\subsubsection{Block partition}
Let $l=t^{\alpha}$, $L=C_{g}t^{\eta+\alpha}$ and 
\begin{eqnarray}\label{eqn:def-m-max}
m_{max} & = & \lceil (1+\epsilon)M\rceil\,.
\end{eqnarray}  
In what follows, we shall upper bound the probability that 
  $H$ chosen according to $\calD(K_{n_1,\ldots,n_d})$
  has a non-crossing 
  hyper-matching of size at least $m_{max}$, i.e.~the probability
  that $L(H)\geq m_{max}$.

We begin with a simple observation; since 
  distinct edges of a non-crossing hyper-matching of $H$ can not have 
  vertices in common,  $L(H)\leq n_i$ for all $i$.
It immediately follows that $L(H)$ is upper bounded by the geometric 
  mean of the $n_i$'s, i.e.~$L(H)\leq N$.
Thus, if $m_{max}>N$, then $\prob{}{L(H)\geq m_{max}}=0$.
This justifies why, in the ensuing discussion, we assume that $m_{max}\leq N$.

Let $J$ be a non-crossing hyper-subgraph of $K_{n_1,\ldots,n_d}$ such that
  the number of edges of $J$ is (exactly equal) $m_{max}$.
We shall partition the edge set of $J$ into consecutive sets of edges to
  which we will refer as \emph{blocks}.
The partition will be such that for any color class, the set of vertices 
  appearing in a block are ``not to far apart'', the precise meaning
  being clarified shortly.
The maximum number of edges in any block will be $s_{max}$, where:
\begin{eqnarray}\label{eqn:def-s-max}
s_{max} = \left\lfloor \frac{l}{N} m_{max}\right\rfloor\,.
\end{eqnarray}
Given two edges $e$ and $\widetilde{e}$ of $J$, such that 
  $e\preceq \widetilde{e}$, we denote by $[e,\widetilde{e}]$
  the collection of edges $f$ of $J$ such that 
  $e\preceq f\preceq \widetilde{e}$.
We now define a partition into blocks of the edge set of $J$,
  denoted $\calP(J)$, as follows:
  $\calP(J)=\set{[e_i,\widetilde{e}_{i}]}{1\leq i\leq q}$
  where the $e_i$'s, the $\widetilde{e}_{i}$'s, and $q$
  are determined through the following process:
\begin{itemize}
\item $e_1$
  is the first (smallest according to $\preceq$) edge of $J$.
\item Assuming $e_i=(v_1^{(i)},v_{2}^{(i)},\ldots,v_{d}^{(i)})$ 
  has already been defined, 
  $\widetilde{e}_{i}=(\widetilde{v}_1^{(i)},\widetilde{v}_{2}^{(i)},\ldots,
  \widetilde{v}_{d}^{(i)})$ 
  is the last edge of $J$ satisfying the following two conditions
  (see Figure~\ref{fig:partition} for an illustration):
\begin{figure}\label{fig:partition}
\begin{center}
\ifpdf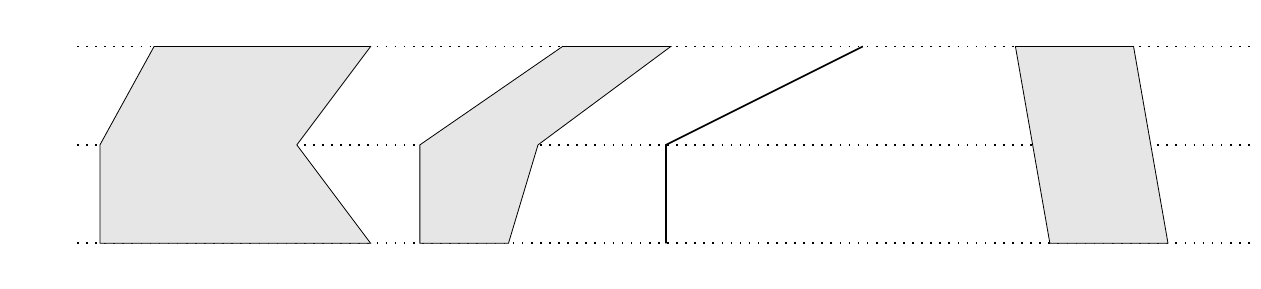\else\input{partition.pstex_t}\fi
\end{center}
\caption{Partition into blocks of a hyper-graph. Each block
  $[e_i,\widetilde{e}_{i}]$ (shown in light grey)
  contains at most $s_{max}$ edges and 
  at most $L$ vertices from each color class.}
\end{figure}
\begin{itemize}
\item $[e_{i},\widetilde{e}_{i}]$ has at most $s_{max}$ elements.
\item $\widetilde{v}_{j}^{(i)}-v_{j}^{(i)}\leq L$ for all
  $1\leq j\leq d$ (where we have relied on the abuse of notation
  entailed by our identification of the $j$-th color class 
  of $K_{n_1,\ldots,n_d}$ with the set $\setof{1,2,\ldots,n_{j}}$
  endowed with the natural order).  
\end{itemize}
\item Assuming $\widetilde{e}_{i}$ has already been defined and 
  provided there are edges $e$ of $J$ strictly larger than 
  $\widetilde{e}_{i}$, we define $e_{i+1}$ to be the smallest such $e$.
\end{itemize}
Clearly, the value taken by $q$ above depends on $J$.
Nevertheless, we will show that the following estimate 
  of $q=\left|\calP(J)\right|$ holds 
  for all $J$ non-crossing hyper-subgraphs of $K_{n_1,\ldots,n_d}$:
\begin{equation}\label{eqn:estimate-q2}\tag{Estimate of $q$}
\frac{N}{l} \ \ \leq \ \ \left|\calP(J)\right| \ \ \leq \ \ \frac{3N}{l}\,.
\end{equation}
Note that each block has at most $s_{max}$ edges and recall that 
  $|E(J)|= m_{max}$. 
Thus, $q\geq m_{max}/s_{max}\geq N/l$.
Now, say a block is \emph{short} if it is either $[e_q,\widetilde{e}_{q}]$
  or a block with exactly $s_{max}$ edges.
Let $I_{0}$ be the collection of indices of short blocks.
It follows that $m_{max} \geq (|I_{0}|-1)s_{max}$.
However, since $m_{max}\geq M=cNt^{-\lambda}$, we know that 
\[
\frac{m_{max}}{s_{max}} 
  \ \ \leq \ \ \frac{N}{l}\cdot\frac{1}{1-N/(lm_{max})}
  \ \ \leq \ \ \frac{N}{l}\cdot\frac{1}{1-t^{\lambda-\alpha}/c}\,.
\]
We thus have, since~\eqref{eqn:constantK} implies 
  that $t^{\lambda-\alpha}< c/2$, that $|I_{0}|\leq 2N/l$.

Say a block is \emph{regular} if it is not short, and let 
  $I_1=[q]\setminus I_{0}$ be the set of indices of such blocks.
We shall call \emph{block cover} the collection of all nodes
  between the first edge of the block (inclusive) and the 
  first edge of the next block (exclusive).
By definition of block partition, if the $i$-th block is regular,
  then for some color class $j$, we must have 
  $v_{j}^{(i+1)}-v_{j}^{(i)}> L$.
Hence, $\sum_{j=1}^{d}(v_{j}^{(i+1)}-v_{j}^{(i)})> L$.
In other words,  a regular block gives rise to 
  a block cover of cardinality at least $L$.
Since every node belongs to at most one block cover, 
  $|I_{1}|\leq S/L$.
Recalling that $L=C_{g}t^{\eta}/l$ and that $S$ satisfies 
  the balance condition (hence, $S\leq C_{g}t^{\eta}N$ for $t\geq t_{g}$), we 
  conclude that $|I_{1}|\leq N/l$.

Putting together the conclusions reached in the last two 
  paragraphs, we see that $q=|I_{0}|+|I_{1}|\leq 3N/l$, which
  establishes the claimed estimate of $q$.

\subsubsection{Partition types}
Let $s_{i}$ be the number of edges of $J$ in the $i$-th block 
  $[e_i,\widetilde{e}_{i}]$ of the partition $\calP(J)$.
Let $q$ be the number of blocks of $\calP(J)$.
We refer to the $(3q)$-tuple $T=(e_1,\widetilde{e}_{1},s_{1},\ldots,
  e_q,\widetilde{e}_{q},s_{q})$ as the \emph{type} of partition
  $\calP(J)$, and denote it $T(\calP(J))$.
Furthermore, let $\calT$ be the collection of all possible types of 
  partitions of hyper-subgraphs of $K_{n_1,\ldots,n_d}$ with 
  exactly $m_{max}$ edges.
\begin{lemma}\label{lem:types-partition}
There is a constant $K_1$, depending only on $d$, such 
  that $\displaystyle\left|\calT\right| 
  \leq \exp\left(K_1\frac{N}{l}\ln(l)\right)$.
\end{lemma}
\begin{proof}
Observe that each $e_i$ is completely determined by specifying its 
  vertices.
Hence, the number of  ways of choosing $e_1,\ldots,e_q$ is at most 
  the number of ways of choosing $q$ elements from each of the 
  node color classes, i.e.~at most $\prod_{i=1}^{d} {n_i \choose q}$.
The number of choices for $\widetilde{e}_1,\ldots,\widetilde{e}_q$ 
  is bounded by the same amount.
On the other hand, since $J$ has exactly $m_{max}$ edges,
  the number of choices for $s_1,\ldots,s_q$ is at most
  the number of ways of summing up to $m_{max}$ with $q$ positive 
  integer summands. 
Since we are assuming that $m_{max}\leq N$ (see comment in this section's
  second paragraph), we have that the aforementioned quantity
  can be bounded by ${N\choose q}$.
Using that ${a\choose b}\leq (ea/b)^{b}$ we obtain, for fixed $q$,
  that the number of types is bounded by
\[
{N\choose q}\left(\prod_{i=1}^{d}{n_i\choose q}\right)^{2}
 \ \ \leq \ 
 \ \left(\frac{eN}{q}\right)^{q}\left(\prod_{i=1}^{d} (en_i/q)\right)^{2q}
 \ \ = \ \ 
  \left(\frac{eN}{q}\right)^{q+2qd}\,.
\]
Recalling our estimate for $q$, we get that
\[
|\calT| 
  \ \ \leq \ \ \sum^{\lfloor 3N/l\rfloor}_{q=\lceil N/l\rceil} 
    \left(\frac{eN}{q}\right)^{q(1+2d)}
  \ \ \leq \ \ \frac{3N}{l}(el)^{3(1+2d)N/l}\,.
\]
Since $\ln(x)\leq x$ for all $x>0$ and by~\eqref{eqn:constant} we know 
  that $l=t^{\alpha}\geq e$, 
\[
\ln |\calT| 
  \ \ \leq \ \ \ln\left(\frac{3N}{l}\right)+(1+2d)\frac{3N}{l}(1+\ln(l))
  \ \ \leq \ \ (2+2d)\frac{3N}{l}(1+\ln(l))
  \ \ \leq \ \ 12(1+d)\frac{N}{l}\ln(l)\,.
\]
The desired conclusion follows choosing $K_1=12(1+d)$.
\end{proof}

\subsubsection{Probability of a block partition occurring} 
The purpose of this section is to show that for a given fixed type $T$, with 
  exponentially small in $M$ probability 
  a hyper-graph chosen according to $\calD(K_{n_1,\ldots,n_d})$
  contains a hyper-subgraph of type $T$ with $m_{max}$ edges.
Specifically, we will prove the following result.
\begin{lemma}\label{lem:probability-partition}
For $T\in\calT$, let $P_{T}$ denote the probability 
  that a hyper-subgraph randomly chosen according to~$\calD(K_{n_1,\ldots,n_d})$
  contains a non-crossing hyper-subgraph  $J$ with $m_{max}$ edges
  such that $T(\calP(J))=T$.
Then, for some absolute constant $K_{2}>0$,
\begin{eqnarray*}
P_{T} & \leq & \exp\left(-K_{2}h\frac{\epsilon^{2}}{1+\epsilon}M\right)\,.
\end{eqnarray*}
\end{lemma}
We now proceed with the proof of the preceding result.
Let $T=(e_1,\widetilde{e}_{1},s_{1},\ldots,
  e_q,\widetilde{e}_{q},s_{q})$.
As before, for all $i$, let $e_{i}=(v_{1}^{(i)}, v_{2}^{(i)},
  \ldots,v_{q}^{(i)})$ and $e_{i}=(\widetilde{v}_{1}^{(i)}, 
  \widetilde{v}_{2}^{(i)}, \ldots,\widetilde{v}_{q}^{(i)})$.
Let $H$ be chosen according to $\calD(K_{n_1,\ldots,n_d})$,
  and let $H_{i}$ be the hyper-subgraph of $H$ induced by the 
  nodes between $e_{i}$ and $\widetilde{e}_{i}$, i.e.,
\[
v_{1}^{(i)}, v_{1}^{(i)}+1,\ldots, \widetilde{v}_{1}^{(i)},
v_{2}^{(i)}, v_{2}^{(i)}+1,\ldots, \widetilde{v}_{2}^{(i)}, 
 \ldots
v_{d}^{(i)}, v_{d}^{(i)}+1,\ldots, \widetilde{v}_{d}^{(i)}\,. 
\]
Note that $H_{i}$ is distributed according to 
  $\calD(K_{n_1^{(i)},n_{2}^{(i)},\ldots,n_{d}^{(i)}})$, 
  where $n_{j}^{(i)}=\widetilde{v}_{j}^{(i)}-v_{j}^{(i)}+1$ is the 
  size of the $j$-th color class of $H_{i}$.
Moreover, if there is a hyper-subgraph $J$ of $H$ such that 
  $T(J)=T$, then it must hold that $L(H_{i})\geq s_{i}$, 
  for all $i=1,\ldots,q$.
Since by hypothesis, $\calD$ satisfies the block independence
  property, the events $L(H_{i})\geq s_{i}$, $i=1,\ldots,q$, are 
  independent, so
\begin{eqnarray*}
P_{T} & \leq & \prod_{i=1}^{q}\prob{}{L\left(\calD(K_{n_1^{(i)},n_{2}^{(i)},
  \ldots,n_{d}^{(i)}})\right)\geq s_{i}}.
\end{eqnarray*}
Now, let $N_{i}$ and $S_{i}$ denote the geometric mean and sum of 
  $n_{1}^{(i)},\ldots,n_{d}^{(i)}$, respectively.
The $i$-th term in the product of the last displayed equation
  will be small provided the sizes of the color classes of $H_{i}$,
  i.e.~the $n_{j}^{(i)}$'s,
  satisfy the size constraints of the definition of a 
  $(c,\lambda,\theta)$-median.
Unfortunately, this may not occur for every $i$, somewhat
  complicating the analysis.
Below we see how to handle this situation.

Since $T(\calP(J))=T$, we know that 
  $n_{1}^{(i)},n_{2}^{(i)},\ldots,n_{d}^{(i)}\leq L$. 
Recalling that $\alpha<\beta$ and applying~\eqref{eqn:constantK}
  we conclude that $S_{i}b\leq dbL 
     = C_{g}dbt^{\eta+\alpha} \leq C_{g}dbt^{\eta+\beta} \leq t^{\theta}$,
  so the size upper bound condition of the definition of 
  a $(c,\lambda,\theta)$-median holds.
However, the same might not be true regarding the 
  size lower bound condition $N_{i}\geq at^{\lambda}$.
In order to handle this situation, we artificially augment the size
  of the blocks where the condition fails.
Specifically, for all $i=1,\ldots,q$ and $j=1,\ldots,d$ we define:
\begin{eqnarray*}
\overline{n}_{j}^{(i)} & = & 
  \max\setof{\delta n_{j}At^{\lambda}/N,n_{j}^{(i)}}\,.
\end{eqnarray*}
As usual, let $\overline{N}_{i}$ and $\overline{S}_{i}$ denote 
  the geometric mean and sum of the $\overline{n}_{j}^{(i)}$'s. 
Now observe that when we augment the sizes of the color classes 
  of the hyper-graphs chosen, by the monotonicity property of 
  random hyper-graph models, the probability of finding 
  a non-crossing hyper-subgraph of size at least $s_i$ increases. 
Hence,
\[
P_{T} 
  \ \ \leq \ \
  \prod_{i=1}^{q} \prob{}{L\left(\calD\left(K_{n_i^{(1)},\ldots,n_i^{(d)}}\right)\right)\geq s_{i}}
  \ \ \leq \ \
  \prod_{i=1}^{q} \prob{}{L\left(\calD\left(K_{\overline{n}_i^{(1)},\ldots,\overline{n}_i^{(d)}}\right)\right)\geq s_{i}}\,.
\]
We claim that the $\overline{N}_{i}$'s and $\overline{S}_{i}$'s satisfy
  the size conditions in the definition of a $(c,\lambda,\theta)$-median.
Indeed, by definition of 
  of $\overline{n}_{j}^{(i)}$, since $n_{j}^{(i)}\leq L$, and
\[
\delta n_{j}\frac{At^{\lambda}}{N} 
  \ \ \leq \ \ 
  \delta n_{j}\frac{t^{\alpha}}{N}
  \ \ \leq \ \
  \delta \frac{St^{\alpha}}{N}
  \ \ \leq \ \
  \delta g(t)t^{\alpha}
  \ \ \leq \ \
  \delta C_{g}t^{\alpha+\eta}
  \ \ = \ \
  \delta L
  \ \ \leq \ \ 
  L\,,
\]
it follows that $\overline{n}_{j}^{(i)}\leq L$, and thence,
  as before augmenting the block sizes, 
  $\overline{S}_{i}b\leq t^{\theta}$.
On the other hand, 
  by definition of $\overline{n}_{j}^{(i)}$, given that
  $N\geq At^{\lambda}$, and since
  by~\eqref{eqn:choiceA-upper} we know that $A\geq a/\delta$,
\[
\overline{N}_{i} 
  \ \ = \ \ 
  \left(\prod_{j=1}^{d}\overline{n}_{j}^{(i)}\right)^{1/d}
  \ \ \geq \ \ 
  \frac{\delta At^{\lambda}}{N}\left(\prod_{j=1}^{d} n_{j}\right)^{1/d}
  \ \ = \ \ 
  \delta At^{\lambda} 
  \ \ \geq \ \ 
  at^{\lambda}\,.
\]
This concludes the proof of the stated claim.

Now, let $\overline{\mediannopar}_{i}$ be a median of 
  $L(\calD(K_{\overline{n}_{i}^{(1)},\ldots,\overline{n}_{i}^{(d)}}))$.
By definition of $(c,\lambda,\theta)$-median,
\[
(1-\delta)c\overline{N}_{i}t^{-\lambda}
  \ \ \leq \ \ 
  \overline{\mediannopar}_{i}
  \ \ \leq \ \
  (1+\delta)c\overline{N}_{i}t^{-\lambda}\,.
\]
Hence, for all $i$ such that $s_{i}\geq (1+\delta)c\overline{N}_{i}t^{-\alpha}
  \geq \overline{\mediannopar}_{i}$, and using that $h$ is a 
  concentration constant for the random model $\calD$, we get
\begin{eqnarray*}
\prob{}{L\left(\calD\left(K_{\overline{n}_{i}^{(1)},\ldots,\overline{n}_{i}^{(d)}}\right)\right) \geq s_{i}} 
  & \leq &
  2\exp\left(-h\frac{(s_{i}-\overline{\mediannopar}_{i})^{2}}{s_{i}}\right) \\
  & \leq & 
  2\exp\left(-h\frac{(s_{i}-(1+\delta)c\overline{N}_{i}t^{-\lambda})^{2}}{s_{i}}\right)\,.
\end{eqnarray*}
Since $s_{i}\leq s_{max}$ for all $i$,
\begin{eqnarray*}
\prob{}{L\left(\calD\left(K_{\overline{n}_{i}^{(1)},\ldots,\overline{n}_{i}^{(d)}}\right)\right) \geq s_{i}} 
  & \leq & 
  2\exp\left(-h\frac{(\max\setof{0,s_{i}-(1+\delta)c\overline{N}_{i}t^{-\lambda}})^{2}}{s_{i}}\right) \\
  & \leq & 
  2\exp\left(-h\frac{(\max\setof{0,s_{i}-(1+\delta)c\overline{N}_{i}t^{-\lambda}})^{2}}{s_{max}}\right)\,.
\end{eqnarray*}
Combining some of the previously derived bounds
\begin{eqnarray*}
-\ln P_{T} 
  & \geq & 
  -\ln\left(\prod_{i=1}^{q}\prob{}{L\left(\calD\left(K_{\overline{n}_{i},\ldots,\overline{n}_{i}^{(d)}}\right)\right)\geq s_{i}}\right) \\
  & \geq & -q\ln(2)+\frac{h}{s_{max}}\sum_{i=1}^{q}
    \left(\max\setof{0,s_{i}-(1+\delta)c\overline{N}_{i}t^{-\lambda}}\right)^{2}\,.
\end{eqnarray*}
We now focus on the summation in the last term in the preceding displayed
  equation.
We lower bound it, via the following generalization of 
  H\"older's Inequality.
\begin{lemma}{[Generalization of H\"older's Inequality]}\label{lem:holder}
For any collection of positive real numbers $x_{i,j}$, $1\leq i\leq q$,
  $1\leq j\leq d$,
\begin{eqnarray*}
\left(\sum_{i=1}^{q}\prod_{j=1}^{d} x_{i,j}\right)^{d}
  & \leq & \prod_{j=1}^{d}\sum_{i=1}^{q} x_{i,j}^{d}\,.
\end{eqnarray*} 
\end{lemma}
Setting $x_{i,j}=(\overline{n}_{j}^{(i)})^{1/d}$ 
  in the aforementioned stated inequality, observing 
  that by definition of $\overline{n}_{j}^{(i)}$
  we have $\overline{n}_{j}^{(i)} \leq n_{j}^{(i)}+\delta n_{j}At^{\lambda}/N$,
  and recalling that the sum of $n_{1}^{(i)},\ldots, n_{d}^{(i)}$
  is at most $n_{j}$,
\[
\sum_{i=1}^{q}\overline{N}_{i} 
  \ \ \leq \ \ 
  \left(\prod_{j=1}^{d}\sum_{i=1}^{q}\overline{n}_{j}^{(i)}\right)^{1/d}
  \ \ \leq \ \
  \left(\prod_{j=1}^{d}\sum_{i=1}^{q}(n_{j}^{(i)}+\delta n_{j}At^{\lambda}/N)\right)^{1/d}
  \ \ \leq \ \
  N(1+\delta qAt^{\lambda}/N)\,.
\]
Because of our estimate for $q$ and~\eqref{eqn:constant}, we conclude that
\[
\sum_{i=1}^{q} \overline{N}_{i} 
  \ \ \leq \ \ N(1+3\delta At^{\lambda-\alpha})
  \ \ \leq \ \ N(1+\delta/3)
  \ \ \leq \ \ N(1+\delta)\,.
\]
By Cauchy-Schwartz's inequality and recalling that the sum of the $s_{i}$'s
  is exactly equal to $m_{max}=\lceil(1+\epsilon)M\rceil$, 
\begin{eqnarray*}
\lefteqn{\sqrt{q\sum_{i=1}^{q}
  \left(\max\setof{0,s_{i}-(1+\delta)c\overline{N}_{i}t^{-\lambda}}\right)^{2}}
  \ \ \geq  \ \
  \sum_{i=1}^{q} 
  \max\setof{0,s_{i}-(1+\delta)c\overline{N}_{i}t^{-\lambda}}} \\
  & & \quad \geq \ \
  m_{max}-(1+\delta)ct^{-\lambda}\sum_{i=1}^{q}\overline{N}_{i} 
  \ \ \geq \ \
  M(1+\epsilon)-M(1+\delta)^{2}\,.
\end{eqnarray*}
Lets now see that the just derived lower bound is actually positive.
Recall, that by definition of $\delta$ we know that $\delta\leq\epsilon/6$
  and $\delta\leq 1$, so
\[
(1+\epsilon)-(1+\delta)^{2} 
  \ \ = \ \ \epsilon -2\delta-\delta^{2}
  \ \ \geq \ \ \epsilon - 3\delta 
  \ \ \geq \ \ \epsilon/2\,.
\]
We then have,
\begin{eqnarray*}
\sqrt{q\sum_{i=1}^{q}
  \left(\max\setof{0,s_{i}-(1+\delta)c\overline{N}_{i}t^{-\lambda}}\right)^{2}}
  & \geq & 
  \frac{\epsilon M}{2}\,.
\end{eqnarray*}
Putting things together, and since $s_{max}\leq (l/N)(1+\epsilon)M$, we find that
\begin{eqnarray*}
\lefteqn{-\ln P_{T} 
  \ \ \geq \ \ -q\ln(2)+\frac{h}{s_{max}}\sum_{i=1}^{q}
  \left(\max\setof{0,s_{i}-(1+\delta)c\overline{N}_{i}t^{-\lambda}}\right)^{2}} \\
  & & \quad \geq 
  \ \ -q\ln(2) + \frac{h}{qs_{max}}\cdot \frac{\epsilon^{2}M^{2}}{4} 
  \ \ \geq \ \ -q\ln(2) + \frac{hN\epsilon^{2}M}{4q(1+\epsilon)l}\,.
\end{eqnarray*}
Finally, recall that by our estimate for $q$ we know that 
  $q\leq 3N/l$ and by~\eqref{eqn:constant} we have that 
  $l=t^{\alpha}\geq 9At^{\lambda}$, so
\[
-\ln P_{T} 
  \ \ \geq \ \ 
  -\frac{N\ln(2)}{3At^{\lambda}}+
    \frac{h\epsilon^{2}M}{12(1+\epsilon)} 
  \ \ = \ \ 
  \frac{cN}{t^{\lambda}}\left(\frac{h\epsilon^{2}}{12(1+\epsilon)}-\frac{\ln(2)}{3Ac}\right)\,.
\]
By~\eqref{eqn:choiceA-upper} we know that 
  $A\geq 8\ln(2)/(hc\delta)$, by definition of $\delta$ we have
  that $\delta\leq\epsilon^{2}/(1+\epsilon)$, implying that
\[
-\ln P_{T} 
  \ \ \geq \ \ 
  \frac{cN}{t^{\lambda}}\left(\frac{h\epsilon^{2}}{12(1+\epsilon)}-\frac{h\delta}{24}\right)
  \ \ \geq \ \ 
  \frac{\epsilon^{2}}{1+\epsilon}\cdot\frac{hM}{24}\,.
\]
We have thus shown that Lemma~\ref{lem:probability-partition}
  holds taking $K_{2}=1/24$.

\subsubsection{Upper tail bound}
We are now ready to finally prove~\eqref{eqn:maintheo-upper} for 
  $N\geq t^{\beta}$.
First, note that
\[
\prob{}{L\left(\calD\left(K_{n_1,\ldots,n_d}\right)\right)\geq m_{max}}
  \ \ \leq \ \ 
  \sum_{T\in\calT} P_{T} 
  \ \ \leq \ \ 
  \left|\calT\right|\cdot\max_{T\in\calT} P_{T}\,.
\]
By Lemmas~\ref{lem:types-partition} 
  and~\ref{lem:probability-partition}, 
  the fact that $l=t^{\alpha}$,
  by our choice of $t_{g}$ so~\eqref{eqn:constantK} 
  would hold, recalling that by definition of $\delta$ we have
  that $\delta\leq \epsilon^{2}/(1+\epsilon)$ and given that 
  $M=cNt^{-\lambda}$,
\begin{eqnarray*}
\lefteqn{\prob{}{L\left(\calD\left(K_{n_1,\ldots,n_d}\right)\right)\geq m_{max}}}
\\
  & & \quad \leq \ \ 
  \exp\left(K_{1}\frac{N}{l}\ln(l)-K_{2}h\frac{\epsilon^{2}}{1+\epsilon}M\right)
  \ \ = \ \  
  \exp\left(K_{1}\alpha\frac{N}{t^{\alpha}}\ln(t)-K_{2}h\frac{\epsilon^{2}}{1+\epsilon}M\right) \\
  & & \quad \leq \ \
  \exp\left(\frac{\delta K_{2}h}{2}\frac{cN}{t^{\lambda}}-K_{2}h\frac{\epsilon^{2}}{1+\epsilon}M\right)
  \ \ \leq \ \
  \exp\left(-\frac{K_{2}h\epsilon^{2}}{2(1+\epsilon)}M\right)\,.
\end{eqnarray*}
We thus conclude that~\eqref{eqn:maintheo-upper} holds for 
  any constant $K\leq K_{2}/2$ (since $K_{2}=1/24$, any $K\leq 1/48$ would do).

\subsection{Upper bounds for the mean and median}\label{sec:upper-mean-median}
We will now establish the two remaining unproved bounds
  claimed in the Main Theorem, i.e.~\eqref{eqn:maintheo-expec} 
  and~\eqref{eqn:maintheo-median}.

Fix $\epsilon=\epsilon_{0}>0$ and 
  choose $\delta$, $A$, $\alpha$, $\beta$, $C_g$, 
  $t_{g}$, $K_1$ and $K_2$ as in Section~\ref{sec:upper-notation}.
We can view $\delta$ as a function
  of $\epsilon$, henceforth denoted $\delta(\epsilon)$.
Similarly, we can view $A$ and $t_{g}$
  as functions of $\delta$, denoted $A(\delta)$ and $t_{g}(\delta)$
  respectively.
Let $A'$ be a sufficiently large constant so
\begin{equation}\label{eqn:definitionA-prime}\tag{Definition of $A'$}
\exp\left(-Kh\frac{\epsilon^{2}_{0}}{4(1+\epsilon_{0}/2)}cA'\right)
  \ \ \leq \ \ \frac{\epsilon_{0}}{28}\,, 
  \quad \mbox{and} \quad
  \frac{7}{6Kh} \ \ \leq \ \ \frac{\epsilon_{0}}{4}cA'\,.
\end{equation}
Also, let $\delta_{0}=\delta(\epsilon_{0}/2)$.
Observe that by definition of $\delta$, for every $\epsilon\geq 6$
  we have that $\delta(\epsilon)=1$.
Define now
  $\widetilde{A}=\max\setof{A(\delta_{0}),A(1),A'}$,
  and $\widetilde{t}_{g}=\max\setof{t_{g}(\delta_{0}),t_{g}(1)}$.

Let $t\geq\widetilde{t}_{g}$ and consider the positive 
  integers $n_1,\ldots,n_d$ with geometric mean $N$ and summing $S$ 
  satisfying the size and balance conditions in the 
  statement of the Main Theorem, i.e.
\begin{equation*}
N \ \ \geq \ \ \widetilde{A}t^{\lambda}\,, 
\qquad \mbox{ and } \qquad
Sb \ \ \leq \ \ g(t)N\,.
\end{equation*}
The choice of $\widetilde{t}_{g}$ and $\widetilde{A}$ guarantee 
  that~\eqref{eqn:maintheo-upper} holds for $\epsilon=\epsilon_{0}/2$
  and for all $\epsilon\geq 6$.

As usual, let $H$ be chosen according to $\calD(K_{n_1,\ldots,n_d})$ and 
  let $M=cNt^{-\lambda}$.
Let $\indic_{A}(x)$ denote the function that takes the value 
  $1$ if $x\in A$ and $0$ otherwise.
Observe that
\begin{eqnarray*}
\lefteqn{\expec{}{L(H)} \ \ = \ \ 
  \expec{}{L(H)\indic_{[0,(1+\epsilon_{0}/2)M)}(L(H))}   } \\
  & & \quad 
  + \ \ \expec{}{L(H)\indic_{[(1+\epsilon_{0}/2)M,7M)}(L(H))}
  + \expec{}{L(H)\indic_{[7M,+\infty]}(L(H))}\,.
\end{eqnarray*}
Lets now upper bound separately 
  each of the terms in the right hand side of the preceding
  displayed equation.
The first one is trivially upper bounded by $(1+\epsilon_{0}/2)M$.
Thanks to~\eqref{eqn:maintheo-upper}, since
  $N\geq \widetilde{A}t^{\lambda}\geq A't^{\lambda}$, and by definition
  of $A'$,
\begin{eqnarray*}
\lefteqn{\expec{}{L(H)\indic_{[(1+\epsilon_{0}/2)M,7M)}(L(H))}
  \ \ \leq \ \
  7M\prob{}{L(H)>(1+\epsilon_{0}/2)M}} \\
  & & \quad \leq \ \
  7M\exp\left(-Kh\frac{\epsilon_{0}^{2}}{4(1+\epsilon_{0}/2)}\cdot\frac{cN}{t^{\lambda}}\right)
  \ \ \leq \ \
  7M\exp\left(-Kh\frac{\epsilon_{0}^{2}}{4(1+\epsilon_{0}/2)}\cdot cA'\right) 
  \ \ \leq \ \
  \frac{M\epsilon_{0}}{4}\,.
\end{eqnarray*}
Now lets consider the third term.
By~\eqref{eqn:maintheo-upper}, since for 
  $\epsilon\geq 6$ it holds that
  $\epsilon/(1+\epsilon)\geq 6/7$,
  given that $M=cNt^{-\lambda}\geq c\widetilde{A}\geq cA'$,
  and by definition of $A'$,
\begin{eqnarray*}
\lefteqn{\expec{}{L(H)\indic_{[7M,+\infty]}(L(H))}
  \ \ = \ \
  \int_{7M}^{\infty} \prob{}{L(H)>t}dt 
  \ \ = \ \
  M\int_{6}^{\infty} \prob{}{L(H)>(1+\epsilon)M}d\epsilon} \\
  & & \quad \leq \ \
  M\int_{6}^{\infty} \exp\left(-Kh\frac{\epsilon^{2}}{1+\epsilon}\cdot M\right)d\epsilon 
  \ \ \leq \ \
  M\int_{6}^{\infty} \exp\left(-\frac{6Kh}{7}\cdot M\epsilon\right)d\epsilon \\
  & & \quad = \ \
  M\left(\frac{6Kh}{7}M\right)^{-1}
  \exp\left(-\frac{36Kh}{7}\cdot M\right) 
  \ \ \leq \ \ 
  M\left(\frac{6Kh}{7}\cdot cA'\right)^{-1} 
  \ \ \leq \ \
  \frac{M\epsilon_{0}}{4}\,.
\end{eqnarray*}  
Summarizing, we have that $\expec{}{L(H)}\leq (1+\epsilon_{0})M$ which
  proves~\eqref{eqn:maintheo-expec}.

Finally, we establish~\eqref{eqn:maintheo-median}.
Again, let $\epsilon>0$ and 
  choose $\delta$, $A$, $\alpha$, $\beta$, $C_g$, 
  $t_{g}$, $K_1$ and $K_2$ as in Section~\ref{sec:upper-notation}.
Let 
\begin{equation}\tag{Definition of $A'$}
A' \ \  = \ \ \max\setof{A,\frac{(1+\epsilon)\ln(2)}{Khc\epsilon^{2}}}\,.
\end{equation}
Now, let $t\geq t_{g}$ and $n_{1},\ldots,n_{d}$ be positive integers 
  with geometric mean $N$ and summing up to $S$ satisfying the 
  size and balance conditions with respect to the just defined
  constant $A'$, i.e.
\[
N \ \ \geq \ \ A't^{\lambda}\,, \quad \mbox{and} \quad
Sb \ \ \leq \ \ g(t)N\,.
\]
By~\eqref{eqn:maintheo-upper} and definition of $A'$, it follows 
  that
\[
\prob{}{L(H)\geq (1+\epsilon)M}
  \ \ \leq \ \ 
  \exp\left(-Kh\frac{\epsilon^{2}}{1+\epsilon}\cdot \frac{cN}{t^{\lambda}}\right)
  \ \ \leq \ \ 
  \exp\left(-Kh\frac{\epsilon^{2}}{1+\epsilon}\cdot cA'\right)
  \ \ \leq \ \ 
  \frac{1}{2}\,.
\]
Hence, every median of $L(H)$ is at most $(1+\epsilon)M$, thus 
  establishing~\eqref{eqn:maintheo-median} and completing the 
  proof of the Main Theorem.

\section{Applications}\label{sec:appl}

\subsection{Random binomial hyper-graph model}\label{sec:random-binomial}
In this section, we show how to apply the Main Theorem to the 
  $d$-partite random binomial hyper-graph model.

We will show that the constant $c$ of the definition of a 
  $(c,\lambda,\theta)$-median for this model is related to 
  a constant that arises in the study of the asymptotic
  behavior of the length of a longest increasing
  subsequence of $d-1$ randomly chosen permutations of $[n]$, when
  $n$ goes to infinity.
We first recall some known facts about this problem.
Given a positive integers $d$ and $n$, consider $d$ permutations
  of $\pi_{1},\ldots,\pi_{d}$ of $[n]$.
We say that 
  $L=\set{(i_{j},\pi_{1}(i_{j}),\ldots,\pi_{d}(i_{j})}{1\leq j\leq \ell}$ is 
  an increasing sequence of $(\pi_1,\ldots,\pi_d)$ of length $\ell$ if
  $i_1< i_2 <\ldots <i_{\ell}$ and 
  $\pi_{t}(i_1) < \pi_{t}(i_2) < \ldots < \pi_{t}(i_{\ell})$
  for $1\leq t\leq d$.
We denote by $lis_{d+1}(n)$ the random variable corresponding to the 
  length of a longest increasing subsequence of $(\pi_1,\ldots,\pi_d)$
  when $\pi_{1},\ldots,\pi_{d}$ are randomly and uniformly chosen.
The study of the asymptotic characteristics of the 
  distribution of $lis_{d}(n)$ will be henceforth referred to as 
  Ulam's problem in $d$ dimensions
  (note that the $d=2$ case corresponds precisely to the 
  setting discussed in the first paragraph of the introductory section
  of this work).

Ulam's problem in $d$-dimensions can be restated geometrically.
Indeed, consider $\vec{x}(1),\ldots,\vec{x}(n)$ uniformly and 
  independently chosen in the $d$-dimensional unit cube $[0,1]^{d}$
  endowed with the natural component wise partial order.
Let $H_{d}(n)$ be the length of a largest chain 
  $C\subseteq\setof{\vec{x}(1),\ldots,\vec{x}(n)}$.
It is not hard to see that $H_{d}(n)$ and $lis_{d}(n)$ follow
  the same distribution.
Bollob\'as and Winkler~\cite{bw88} showed that for every $d$ there exists
  a constant $c_{d}$ such that 
  $H_{d}(n)/\sqrt[d]{n}$ (and thus also 
  $lis_{d}(n)/\sqrt[d]{n}$) goes to $c_{d}$ as $n\to\infty$.
Only the values $c_1=1$ and $c_2=2$ are known for these constants.
However, in~\cite{bw88} it is shown that 
  $c_{i}\leq c_{i+1}$ and $c_{i}< e$ for all $i$, and that 
  the $\lim_{d\to\infty} c_{d} =e$.

\medskip
Now, back to our problem. 
Our immediate goal is to estimate a median of 
  $L(\calG(K_{n_1,\ldots,n_d},p))$.
Consider~$H$ chosen according to $\calG(K_{n_1,\ldots,n_d},p)$
  and let $H'$ be the hyper-subgraph of $H$ obtained from $H$ after removal
  of all edges incident to nodes of degree at least $2$.
Let $E=E(H)$ and $E'=E(H')$.
In order to approximate a median of 
  $L(\calG(K_{n_1,\ldots,n_d},p))$ it will be useful to estimate
  first the expected value of $L(H')$.
We now come to a crucial observation: $L(H')$ is precisely 
  the length of a largest chain (for the natural order among
  edges) contained in $E'$, or equivalently the length of a 
  longest increasing subsequence of $d-1$ permutations of 
  $\setof{1,\ldots,|E'|}$.
The preceding observation will enable us to build on the
  known results concerning Ulam's problem and use them 
  in the analysis of the Longest Non-crossing Matching
  problem for the random binomial hyper-graph model.
In particular, the following concentration result due to Bollob\'as and 
  Brightwell~\cite{bb92} for the length of a 
  $d$-dimensional longest increasing subsequence will be 
  useful for our purposes.
\begin{theorem}{[Bollob\'as and Brightwell~\cite[Theorem 8]{bb92}]}\label{th:bb}
For every  $d\geq 2$, there is a constant $D_{d}$ such that
  for $m$ sufficiently large and $2<\lambda < m^{1/2d}/\log\log m$,
\begin{eqnarray*}
\prob{}{\left| lis_{d}(m)-\expec{}{lis_{d}(m)} \right| > 
  \frac{\lambda D_{d}m^{1/2d}\log(m)}{\log\log(m)}}
  & \leq & 80\lambda^{2}e^{-\lambda^2}\,.
\end{eqnarray*}
\end{theorem}
We will not directly apply the preceding result. 
Instead, we rely on the following:
\begin{corollary}\label{cor:lis-concentration}
For every $d\geq 2$, $t>0$ and $\alpha>0$, there is a $m_{0}(t,\alpha,d)$ 
  sufficiently large such that if $m\geq m_{0}$, then
\begin{eqnarray*}
\prob{}{\left|lis_{d}(m)-c_{d}m^{1/d} \right| > tc_{d}m^{1/d}}
  & \leq & \alpha\,.
\end{eqnarray*}
\end{corollary}  
\begin{proof}
Let $D_d$ be the constant in the statement of Theorem~\ref{th:bb}.
By definition of Ulam's constant, 
  we know that $\lim_{n\to\infty} \expec{}{lis_{d}(m)}/\sqrt[d]{m}=c_{d}$.
Hence, we can choose $m_{0}=m_{0}(t,\alpha,d)$ sufficiently large
  so that for all $m\geq m_{0}$, 
  Theorem~\ref{th:bb} holds and in addition the following 
  conditions are satisfied:
\begin{itemize}
\item $\left|\expec{}{lis_{d}(m)}\right|-c_{d}m^{1/d} < tc_{d}m^{1/d}/2$.

\item $\lambda = \lambda(m) \eqdef \frac{tc_{d}}{2D_{d}}\cdot 
  \frac{m^{1/2d}\log\log(m)}{\log(m)} \leq \frac{m^{1/2d}}{\log\log(m)}$
  and $80\lambda^{2}e^{-\lambda^{2}} \leq \alpha$.
\end{itemize}
(Both conditions can be satisfied since $(\log\log(m))^{2}=o(\log(m))$
  and given that $\lambda(m)\to\infty$ when $m\to\infty$.)
It follows that for all $m>m_{0}$,
\begin{eqnarray*}
\lefteqn{\prob{}{\left|lis_{d}(m)-c_{d}m^{1/d}\right|> tc_{d}m^{1/d}}
  } \\ 
  & \leq &   \prob{}{\left|lis_{d}(m)-\expec{}{lis_{d}(m)}\right| 
    +\left|\expec{}{lis_{d}(m)}-c_{d}m^{1/d}\right|>tc_{d}m^{1/d}} \\
  & \leq & 
  \prob{}{\left|lis_{d}(m)-\expec{}{lis_{d}(m)}\right|
    > \frac{1}{2}tc_{d}m^{1/d}} \\
  & = & 
  \prob{}{\left|lis_{d}(m)-\expec{}{lis_{d}(m)}\right|
    > \frac{\lambda D_{d}m^{1/2d}\log(m)}{\log\log(m)}}  \\
  & \leq &
  80\lambda^{2}e^{-\lambda^{2}}\,.
\end{eqnarray*}
\end{proof}
For future reference, we recall a well known variant of Chebyshev's 
  inequality.
\begin{proposition}{[Chebyshev's inequality for indicator random variables]}
Let $X_{1},\ldots,X_{m}$ be random variables taking values in $\setof{0,1}$ and
  let $X$ denote $X_{1}+\ldots+X_{m}$.
Also, let $\Delta=\sum_{i,j:i\neq j} \expec{}{X_{i}X_{j}}$.
Then, for all $t\geq 0$,
\begin{eqnarray*}
\prob{}{\left|X-\expec{}{X}\right| \geq t} & \leq & 
  \frac{1}{t^{2}}\left(\expec{}{X}(1-\expec{}{X})+\Delta\right)\,.
\end{eqnarray*}
Moreover, if $X_{1},\ldots,X_{m}$ are independent, then
\begin{eqnarray*}
\prob{}{\left|X-\expec{}{X}\right| \geq t} & \leq & 
  \frac{\expec{}{X}}{t^{2}}\,.
\end{eqnarray*}
\end{proposition}
\begin{proof}
Observe that since $X_{i}$ is an indicator variable, then 
  $\expec{}{X^{2}_{i}}=\expec{}{X_{i}}$.
Thus, if we let $\var{}{X}$ denote the variance of $X$,
\[
\var{}{X} 
  \ \ = \ \
  \expec{}{X^{2}}-(\expec{}{X})^{2}
  \ \ = \ \
  \sum^{m}_{i=1}\expec{}{X^{2}_{i}} + \Delta - (\expec{}{X})^{2}
  \ \ = \ \ 
  \expec{}{X}(1-\expec{}{X})+\Delta\,.
\]
A direct application of Chebyshev's inequality yields the first 
  bound claimed.
The second stated bound, follows from the first one and the 
  fact that if $X_{1},\ldots,X_{m}$ are independent, then
  $\Delta \leq (\expec{}{X})^{2}$.
\end{proof}

We will also need the following two lemmas.
\begin{lemma}\label{lem:fact}
Let $N$ and $S$ denote the geometric mean and sum of 
  $n_1,\ldots,n_d$.
If $\widetilde{N}=\left(\prod_{j=1}^{d}(n_j-1)\right)^{d}$,
  then $N^{d}-\widetilde{N}^{d}\leq S^{d-1}$.
\end{lemma}
\begin{proof}
Direct application of Lemma~\ref{lem:holder}.
\end{proof}

\begin{lemma}\label{lem:binomial-prep}
Let $N$ and $S$ denote the geometric mean and sum of 
  $n_1,\ldots,n_d$.
If $\widetilde{N}=\left(\prod_{j=1}^{d}(n_j-1)\right)^{d}$, 
  then the following hold:
\begin{eqnarray}
\expec{}{|E|} & = & N^{d}p\,, \label{eqn:binomial-first} \\
\expec{}{|E'|} & = & N^{d}p(1-p)^{N^d-\widetilde{N}^{d}}
  \ \ \geq \ \ N^{d}p(1-S^{d-1}p)\,,
  \label{eqn:binomial-second} \\
\expec{}{|E\setminus E'|} & \leq & N^{d}S^{d-1}p^{2}\,. 
  \label{eqn:binomial-third}
\end{eqnarray}
Moreover, for all $\eta>0$, 
\begin{eqnarray}\label{eqn:binomial-fourth}
\prob{}{|E|-\expec{}{|E|}\geq \eta\expec{}{|E|}}
  & \leq & \frac{1}{\eta^{2}\expec{}{|E|}}\,.
\end{eqnarray}
\end{lemma}
\begin{proof}
Let $K=K_{n_1,\ldots,n_k}$, and for each $e\in E(K)$ let $X_e$ and 
  $Y_e$ denote the indicators of the events $e\in E$ and $e\in E'$,
  respectively.
Note that $|E|=\sum_{e\in E(K)} X_{e}$ and $|E'|=\sum_{e\in E(K)} Y_{e}$.
Clearly, $\expec{}{X_{e}}=p$ for all $e\in E(K)$.
Moreover, $e\in E'$ if and only if $e\in E$ and no edge 
  $f\in E\setminus \setof{e}$ intersects $e$.
Since the number of edges in $E(K)$ that intersect any given $e\in E(K)$
  is exactly $N^{d}-\widetilde{N}^{d}$, we have that
  $\expec{}{Y_{e}}=p(1-p)^{N^{d}-\widetilde{N}^{d}}$.
Observing that $|E(K)|=N^{d}$ we obtain~\eqref{eqn:binomial-first}
  and the first equality in~\eqref{eqn:binomial-second}.
On the other hand, since $(1-p)^{m}\geq 1-pm$ and by 
  Lemma~\ref{lem:fact}, we can 
  finish the proof of~\eqref{eqn:binomial-second} by noting that
\[
\expec{}{|E'|} 
  \ \ = \ \ 
  N^{d}p(1-p)^{N^{d}-\widetilde{N}^{d}}
  \ \ \geq \ \ 
  N^{d}p(1-(N^{d}-\widetilde{N}^{d})p) 
  \ \ \geq \ \ 
  N^{d}p(1-S^{d-1}p)\,.
\]
Inequality~\eqref{eqn:binomial-third} is a consequence 
  of~\eqref{eqn:binomial-first}, \eqref{eqn:binomial-second}, 
  and the fact that $E'\subseteq E$, as follows:
\[
\expec{}{|E\setminus E'|} 
  \ \ = \ \ \expec{}{|E|-|E'|}
  \ \ \leq \ \ 
  N^{d}S^{d-1}p^{2}\,,
\]
Applying Chebyshev's inequality for 
  independent indicator random variables
  $\set{X_{e}}{e\in E(K)}$ yields~\eqref{eqn:binomial-fourth}.
\end{proof}
We are now ready to exploit the fact, already mentioned, that $L(H')$
  equals the length of a longest increasing subsequence of 
  $d-1$ permutations of $\setof{1,\ldots,|E'|}$, and then
  apply Corollary~\ref{cor:lis-concentration} in order to 
  estimate its value.
Formally, we prove the following claim.
\begin{proposition}\label{prop:binomial}
Let $\delta > 0$, $d\geq 2$, and 
  $N$ and $S$ be the geometric mean and sum of 
  positive integers $n_1,\ldots,n_d$, respectively.
Moreover, let $M=c_{d}Np^{1/d}$ where $c_d$ is the $d$-dimensional Ulam
  constant.
Then, there is a constant $C=C(\delta)$ sufficiently large such
  that:
\begin{itemize}
\item If $Np^{1/d}\geq C$ and $12S^{2d-2}p^{2-1/d}\leq d^{d-1}\delta c_{d}$, 
  then every median of $L(\calG(K_{n_1,\ldots,n_d},p))$ is 
  at most $(1+\delta)M$.

\item If $Np^{1/d}\geq C$ and $12S^{d-1}p\leq \delta$, then
  every median of $L(\calG(K_{n_1,\ldots,n_d},p))$ is at least $(1-\delta)M$.
\end{itemize}
\end{proposition}
\begin{proof}
To prove that every median of $L(\calG(K_{n_1,\ldots,n_d},p))$ is 
  at most $(1+\delta)M$, it suffices to show that 
  $\prob{}{L(H)\geq (1+\delta)M}$ is at most $1/2$.
To establish the latter, note that $L(H)\leq L(H')+|E\setminus E'|$, hence
\begin{eqnarray*}
\prob{}{L(H)\geq (1{+}\delta)M}
  & \leq & 
   \prob{}{|E\setminus E'|\geq \frac{M\delta}{2}}
   +\prob{}{L(H')\geq (1+\delta/2)M} \\
  & \leq &
  \prob{}{|E\setminus E'|\geq \frac{M\delta}{2}}
  +\prob{}{|E'|\geq (1+\delta/2)\frac{M^{d}}{c_{d}^{d}}} \\
  & & \quad
  +\prob{}{L(H')\geq (1+\delta/2)M, |E'|<(1+\delta/2)\frac{M^{d}}{c_{d}^{d}}}\,.
\end{eqnarray*}
We now separately upper bound each of the latter three terms.
For the first one, we rely on Markov's inequality, 
  inequality~\eqref{eqn:binomial-second}
  of Lemma~\ref{lem:binomial-prep}, 
  the fact that $N\leq S/d$, 
  and our hypothesis, to conclude that:
\begin{eqnarray*}
\lefteqn{\prob{}{|E\setminus E'|\geq \frac{M\delta}{2}}
  \ \ \leq \ \ 
  \frac{2}{M\delta}\expec{}{|E\setminus E'|}
  \ \ \leq \ \ 
  \frac{2N^{d}S^{d-1}p^{2}}{\delta c_{d}Np^{1/d}} } \\
  & & \quad
  \ \ = \ \ 
  \frac{2N^{d-1}S^{d-1}p^{2-1/d}}{\delta c_{d}}
  \ \ \leq \ \ 
  \frac{2S^{2d-2}p^{2-1/d}}{d^{d-1}\delta c_{d}}
  \ \ \leq \ \ 
  \frac{1}{6}\,.
\end{eqnarray*}
To bound the second term, note that 
  $|E|\geq |E'|$, and recall~\eqref{eqn:binomial-first} 
  and~\eqref{eqn:binomial-fourth} of 
  Lemma~\ref{lem:binomial-prep}, so
\[
\prob{}{|E'|\geq (1+\delta/2)\frac{M^{d}}{c_{d}^{d}}}
  \ \ = \ \
  \prob{}{|E|\geq (1+\delta/2)\expec{}{|E|}}
  \ \ \leq \ \
  \frac{4}{\delta^{2}\expec{}{|E|}}
  \ \ = \ \
  \frac{4}{\delta^{2}N^{d}p}\,.  
\]
Since by assumption $N^{d}p\geq C^{d}$, it suffices to take
  $C^{d}\geq 24/\delta^{2}$ in order to derive an upper 
  bound of $1/6$ for the second term.

Finally, we focus on the third term. 
Let $m=\lfloor (1+\delta/2)M^{d}/c_{d}^{d}\rfloor$.
Recall that conditioned on $|E'|=n'$, the random
  variable $L(H')$ follows the same distribution 
  as $lis(n')$.
Thus, since $n\geq n'$ implies that $lis(n)$ dominates $lis(n')$,
  and given that $(1+x)^{a}\leq 1+ax$ for $x\geq -1$ and $0< a < 1$,
\begin{eqnarray*}
\lefteqn{\prob{}{L(H')\geq (1+\delta/2)M,|E'|<(1+\delta/2)\frac{M^{d}}{c_{d}^{d}}}
  \ \ \leq \ \ 
  \prob{}{lis_{d}(m)\geq\frac{1+\delta/2}{(1+\delta/2)^{1/d}}c_{d}m^{1/d}} } \\
  & \leq & 
  \prob{}{lis_{d}(m)\geq\frac{1+\delta/2}{1+\delta/(2d)}c_{d}m^{1/d}} 
  \ \ = \ \
  \prob{}{lis_{d}(m)\geq\left(1+\frac{(d-1)\delta}{2d+\delta}\right)c_{d}m^{1/d}}\,.
\end{eqnarray*}  
Setting $t=(d-1)\delta/(2d+\delta)$ and requiring that 
  $C^{d}\geq m_{0}+1$ with $m_{0}=m_{0}(t,1/6,d)$ as in 
  Corollary~\ref{cor:lis-concentration}, and since by assumption
  $N^{d}p\geq C^{d}$, we have
\[
m \ \ = \ \ \lfloor (1+\delta/2)M^{d}/c_{d}^{d}\rfloor
  \ \ = \ \ \lfloor (1+\delta/2)N^{d}p\rfloor
  \ \ \geq \ \ \lfloor C^{d} \rfloor 
  \ \ \geq \ \ m_{0}\,.
\]
Thus, we can apply Corollary~\ref{cor:lis-concentration} and conclude that
\[
\prob{}{L(H')\geq (1+\delta/2)M,|E'|<(1+\delta/2)M^{d}/c_{d}^{d}}
  \ \ \leq \ \ \frac{1}{6}\,.
\]
In summary, $\prob{}{L(H)\geq (1+\delta)M} \leq 3(1/6) = 1/2$ as 
  we wanted to show.

Now, to prove that every median of $L(\calG(K_{n_1,\ldots,n_d},p))$ is 
  at least $(1-\delta)M$, it suffices to show that 
  $\prob{}{L(H)\leq (1-\delta)M}$ is at most $1/2$.
Note that $L(\cdot)$ is non-negative, so 
  we can always assume that $\delta \leq 1$.
Since $L(H')\leq L(H)$,
\begin{eqnarray*}
\prob{}{L(H) \leq (1-\delta)M}
  & \leq &
  \prob{}{|E|\leq (1-\delta)\frac{M^{d}}{c_{d}^{d}}}
   + \prob{}{L(H')\leq (1-\delta)M, |E|>(1-\delta)\frac{M^{d}}{c_{d}^{d}}} \\
  & \leq & 
  \prob{}{|E|\leq (1-\delta)\frac{M^{d}}{c_{d}^{d}}}
   + \prob{}{|E\setminus E'|\geq (\delta/2)\frac{M^{d}}{c_{d}^{d}}} \\
  & & \quad
   + \prob{}{L(H')\leq (1-\delta)M, |E'|>(1-\delta/2)\frac{M^{d}}{c_{d}^{d}}}
\end{eqnarray*}
As above, we separately bound each of the two latter terms.
In the case of the first term, by~\eqref{eqn:binomial-first}
  and~\eqref{eqn:binomial-fourth} of Lemma~\ref{lem:binomial-prep},
\[
\prob{}{|E|\leq (1-\delta)\frac{M^{d}}{c_{d}^{d}}}
  \ \ = \ \
  \prob{}{|E| \leq (1-\delta)\expec{}{|E|}} 
  \ \ \leq \ \
  \frac{1}{\delta^{2}\expec{}{|E|}}
  \ \ = \ \ 
  \frac{1}{\delta^{2}N^{d}p}\,.
\]
Since by assumption $N^{d}p\geq C^{d}$, it suffices to take 
  $C^{d}\geq 6/\delta^{2}$ in order to establish an upper bound
  of $1/6$ for the term under consideration.

To bound the second term, simply apply Markov's inequality, 
  use~\eqref{eqn:binomial-third} of Lemma~\ref{lem:binomial-prep},
  and recall that by assumption $12pS^{d-1}\leq\delta$ ---
  an upper bound of $1/6$ follows for the term
  under consideration.

Now, for the third term, let $m=\lceil (1-\delta/2)M^{d}/c_{d}^{d}\rceil$.
Recall that conditioned on $|E'|=n'$, the random
  variable $L(H')$ follows the same distribution 
  as $lis(n')$.
Thus, since $n'\geq n$ implies that $lis(n')$ dominates $lis(n)$,
  some basic arithmetic and given that 
  $(1+x)^{a}\leq 1+ax$ for $x\geq -1$ and $0< a < 1$,
\begin{eqnarray*}
\lefteqn{\prob{}{L(H')\leq (1-\delta)M,|E'|>(1-\delta/2)\frac{M^{d}}{c_{d}^{d}}}
  \ \ \leq \ \ \prob{}{lis_{d}(m) \leq (1-\delta)M}} \\
  & & \quad \leq \ \ \prob{}{lis_{d}(m) \leq 
    \frac{1-\delta}{(1-\delta/2)^{1/d}}c_{d}m^{1/d}} 
  \ \ \leq \ \ 
  \prob{}{lis_{d}(m) \leq (1-\delta/2)^{1-1/d}c_{d}m^{1/d}} \\
  & & \quad \leq \ \ 
  \prob{}{lis_{d}(m) \leq \left(1-\frac{\delta}{2}\left(1-\frac{1}{d}\right)\right)c_{d}m^{1/d}}\,.
\end{eqnarray*}
Setting $t=(\delta/2)(1-1/d)$, requiring that 
  $C\geq (m_{0}/(1-\delta/2))^{1/d}$
  with $m_{0}=m_{0}(t,1/6,d)$ as in Corollary~\ref{cor:lis-concentration},
  and since by assumption $N^{d}p\geq C^{d}$, we get
\[
m \ \ \geq \ \ (1-\delta/2)\frac{M^{d}}{c_{d}^{d}}
  \ \ = \ \ (1-\delta/2)N^{d}p
  \ \ \geq \ \ (1-\delta/2)C^{d} 
  \ \ \geq \ \ m_{0}\,.
\]  
Thus, we can apply Corollary~\ref{cor:lis-concentration} and conclude that
  the third term is also upper bounded by $1/6$.

Summarizing, $\prob{}{L(H)\leq (1-\delta)M} \leq 3(1/6)=1/2$ as
  we wanted to show.
\end{proof}

\begin{corollary}\label{cor:random-binomial-median}
Let $d\geq 2$.
If $t=1/p$, then the model $(\calG(K_{n_1,\ldots,n_d},p))$ of internal parameter
  $t$ admits a $(c,\lambda,\theta)$-median where
\begin{eqnarray*}
(c,\lambda,\theta) & = & \left(c_d,\frac{1}{d},\frac{2d-1}{2d(d-1)}\right)\,.
\end{eqnarray*}
\end{corollary}
\begin{proof}
As usual, let $N$ and $S$ denote the geometric mean and sum 
  of $n_1,\ldots,n_d$.
Let $H$ be chosen according to $\calG(K_{n_1,\ldots,n_d},p)$,
  $M=cN/t^{\lambda}=c_dNp^{1/d}$, $\delta>0$, and 
  $C(\delta)$ be as in Proposition~\ref{prop:binomial}.
Define $a(\delta)=C(\delta)$, $b(\delta)=(12/(\delta d^{d-1}c_{d}))^{1/(2d-2)}$
  and $t'(\delta)$ sufficiently large so $t>t'(\delta)$ and
  $t^{1-1/(2d)} < (\delta/12)t(b(\delta))^{d-1}$.
Note that if $t>t'(\delta)$, $N\geq a(\delta)t^{1/d}$, and
  $Sb(\delta)\leq t^{(2d-1)/(2d(d-1))}$, then the hypothesis 
  of Proposition~\ref{prop:binomial} will be satisfied, and thence
  every median of $L(H)$ will be between $(1-\delta)M$ 
  and $(1+\delta)M$.
\end{proof}
Recalling that by Proposition~\ref{prop:model-concentration-constants}
  we know that $h=1/4$ is a concentration constant for 
  the $d$-dimensional binomial random hyper-graph model,   
  by Corollary~\ref{cor:random-binomial-median} and the 
  Main Theorem, we obtain the following:
\begin{theorem}\label{th:random-binomial}
Let $\epsilon>0$ and $g:\RR\to\RR$ be such that $g(t)=O(t^{\eta})$
  for a given $0\leq\eta < 1/(2d(d-1))$. 
Fix $n_1,\ldots,n_d$ and let $N$ and $S$ denote their geometric
  mean and sum, respectively. 
There exists a sufficiently small $p_0$ and sufficiently large 
  $A$ such that if $p\leq p_{0}$, $Np^{1/d}\geq A$ and $S\leq g(1/p)N$, 
  then for $M=c_dNp^{1/d}$ where $c_d$ is the $d$-dimensional Ulam
  constant,
\[
(1-\epsilon)M 
  \ \ \leq \ \ 
  \expec{}{L(\calG(K_{n_1,\ldots,n_d},p))}
  \ \ \leq \ \ 
  (1+\epsilon)M\,,
\]
and the following hold:
\begin{itemize}
\item If $\median{}{L(\calG(K_{n_1,\ldots,n_d},p))}$ is a median
  of $L(\calG(K_{n_1,\ldots,n_d},p))$,
\[
(1-\epsilon)M 
  \ \ \leq \ \ 
  \median{}{L(\calG(K_{n_1,\ldots,n_d},p))}
  \ \ \leq \ \ 
  (1+\epsilon)M\,.
\]

\item 
There is an absolute constant $C>0$ such that
\begin{eqnarray*}
\prob{}{L(\calG(K_{n_1,\ldots,n_d},p))\leq (1-\epsilon)M}
  & \leq & \exp\left(-C\epsilon^{2}M\right)\,, \\
\prob{}{L(\calG(K_{n_1,\ldots,n_d},p))\geq (1+\epsilon)M}
  & \leq & \exp\left(-C\frac{\epsilon^{2}}{1+\epsilon}M\right)\,.
\end{eqnarray*}
\end{itemize}
\end{theorem}
We are now ready to prove Theorem~\ref{th:main-binomial-model}
  which is this section's main result, and was already stated in the 
  main contributions section.
\begin{prooff}{Theorem~\ref{th:main-binomial-model}}
Let $n,n',n''$ be positive integers such that $n=n'+n''$.
Clearly, 
\begin{eqnarray*}
\expec{}{L(\calG(K^{(d)}_{n},p))} & \geq &
  \expec{}{L(\calG(K^{(d)}_{n'},p))}+\expec{}{L(\calG(K^{(d)}_{n''},p))}\,.
\end{eqnarray*}
By subadditivity, it follows that the limit of 
  $\expec{}{L(\calG(K^{(d)}_{n},p))}$ when normalized by $n$ 
  exists and 
  equals $\delta_{p}=\inf_{n\in\NN} \expec{}{L(\calG(K^{(d)}_{n},p))/n}$.
A direct application of Theorem~\ref{th:random-binomial}
  yields that 
  $\delta_{p}/\sqrt[d]{p}\to c_{d}$ when $p\to 0$.
\end{prooff}

\subsection{Random word model}\label{sec:random-word}
In this section, we consider the random $d$-word model.
The structure, arguments and type of derived results are similar
  to those obtained in the preceding section.
However, the intermediate calculations are somewhat longer 
  and more involved. 
We omit the proofs of this section's results from the current
  draft.

As in the preceding section, we first show that the 
  random model under consideration admits
  a $(c,\lambda,\theta)$-median.
Now consider~$H$ chosen according to $\Sigma(K_{n_1,\ldots,n_d},k)$
  and let $H'$ be the hyper-subgraph of $H$ obtained from $H$ 
  as in the preceding section (i.e.~by removal
  of all edges incident to nodes of degree at least $2$).
Let $E=E(H)$ and $E'=E(H')$.
For the random word model, the analogue
  of Lemma~\ref{lem:binomial-prep} is the following:
\begin{lemma}\label{lem:word-prep}
Let $N$ and $S$ be the geometric mean and sum of 
  positive integers $n_1,\ldots,n_d$, respectively.
Then,
\begin{eqnarray}
\expec{}{|E|} & = & \frac{N^{d}}{k^{d-1}}\,, \\
\expec{}{|E'|} & = & \frac{N^{d}}{k^{d-1}}\left(\frac{k-1}{k}\right)^{S-d}
  \ \ \geq \ \ \frac{N^{d}}{k^{d-1}}\left(1-\frac{S}{k}\right)\,, \\
\expec{}{|E\setminus E'|} & \leq & \frac{N^{d}S}{k^{d}}\,.
\end{eqnarray}
Moreover, for all $\eta>0$, 
\begin{eqnarray*}
\prob{}{|E'|-\expec{}{|E'|}\geq \eta\expec{}{|E'|}}
  & \leq & 
  \frac{1}{\eta^{2}\expec{}{|E'|}}
    + \frac{1}{\eta^{2}}\left(\left(\frac{k-1}{k-2}\right)^{2d-1}-1\right)\,.
\end{eqnarray*}
\end{lemma}
We can now determine an estimate the median of $L(\Sigma(K_{n_1,\ldots,n_d},k)$.
\begin{proposition}
Let $\delta > 0$, $d\geq 2$, and 
  $N$ and $S$ be the geometric mean and sum of 
  positive integers $n_1,\ldots,n_d$, respectively.
Moreover, let $M=c_{d}N/k^{1-1/d}$ where $c_d$ is the $d$-dimensional Ulam
  constant.
Then, there are sufficiently large constants $C=C(\delta)$ and $K=K(\delta)$
  such that:
\begin{itemize}
\item If $k\geq K$, $N\geq Ck^{1-1/d}$, $12S^{d} \leq \delta c_{d}k^{d-1+1/d}$, 
  and $S\leq k/2$, then any median of $L(\Sigma(K_{n_1,\ldots,n_d},k)$ is 
  upper bounded by $(1+\delta)M$.

\item If $k\geq K$, $N\geq Ck^{1-1/d}$, and $S\leq \delta k/2$,
  then every median of 
  $L(\Sigma(K_{n_1,\ldots,n_d},k)$ is at least $(1-\delta)M$.
\end{itemize}
\end{proposition}

\begin{corollary}\label{cor:random-word-median}
The model $(\Sigma(K_{n_1,\ldots,n_d},k))$ of internal parameter
  $k$ admits a $(c,\lambda,\theta)$-median where
\begin{eqnarray*}
(c,\lambda,\theta) & = & \left(c_d,1-\frac{1}{d},1-\frac{1}{d}+\frac{1}{d^{2}}\right)\,.
\end{eqnarray*}
\end{corollary}
Recalling that by Proposition~\ref{prop:model-concentration-constants}
  we have that $h=1/(4d)$ is a concentration constant for 
  the random $d$-word model,   
  by the preceding corollary and the 
  Main Theorem, we obtain the following:
\begin{theorem}\label{th:random-word}
Let $\epsilon>0$ and $g:\RR\to\RR$ be such that $g(k)=O(k^{\eta})$
  for a given $0\leq\eta < 1/d^{2}$. 
Fix $n_1,\ldots,n_d$ and let $N$ and $S$ denote their geometric
  mean and sum, respectively. 
There exists sufficiently large constants $k_0$ and 
  $A$ such that if $k\leq k_{0}$, $Nk^{1-1/d}\geq A$ and $S\leq g(k)N$, 
  then for $M=c_dN/k^{1-1/d}$ where $c_d$ is the $d$-dimensional Ulam
  constant,
\[
(1-\epsilon)M 
  \ \ \leq \ \ 
  \expec{}{L(\Sigma(K_{n_1,\ldots,n_d},k))}
  \ \ \leq \ \ 
  (1+\epsilon)M\,,
\]
and the following hold:
\begin{itemize}
\item If $\median{}{L(\Sigma(K_{n_1,\ldots,n_d},k))}$ is a median
  of $L(\Sigma(K_{n_1,\ldots,n_d},k))$,
\[
(1-\epsilon)M 
  \ \ \leq \ \ 
  \median{}{L(\Sigma(K_{n_1,\ldots,n_d},k))}
  \ \ \leq \ \ 
  (1+\epsilon)M\,.
\]

\item 
There is an absolute constant $C>0$ such that
\begin{eqnarray*}
\prob{}{L(\Sigma(K_{n_1,\ldots,n_d},k))\leq (1-\epsilon)M}
  & \leq & \exp\left(-\frac{C}{d}\epsilon^{2}M\right)\,, \\
\prob{}{L(\Sigma(K_{n_1,\ldots,n_d},k))\geq (1+\epsilon)M}
  & \leq & \exp\left(-\frac{C}{d}\frac{\epsilon^{2}}{1+\epsilon}M\right)\,.
\end{eqnarray*}
\end{itemize}
\end{theorem}
We are now ready to prove Theorem~\ref{th:main-word-model}
  which is this section's main result, and was already stated in the 
  main contributions section.
\begin{prooff}{Theorem~\ref{th:main-word-model}}
Let $n,n',n''$ be positive integers such that $n=n'+n''$.
Clearly, 
\begin{eqnarray*}
\expec{}{L(\Sigma(K^{(d)}_{n},k))} & \geq &
  \expec{}{L(\Sigma(K^{(d)}_{n'},k))}+\expec{}{L(\Sigma(K^{(d)}_{n''},k))}\,.
\end{eqnarray*}
By subadditivity, it follows that the limit of 
  $\expec{}{L(\Sigma(K^{(d)}_{n},k))}$ when normalized by $n$ 
  exists and equals $\gamma_{k}=\inf_{n\in\NN}\expec{}{L(\Sigma(K^{(d)}_{n},k))/n}$.
A direct application of Theorem~\ref{th:random-word}
  yields that $k^{1-1/d}\gamma_{k}\to c_{d}$ when $k\to\infty$.
\end{prooff}

\subsection{Symmetric and anti-symmetric binomial random graph models}\label{sec:symmetric}
Throughout this section we focus on the study of 
  $L(\calD)$ when $\calD$ is either $S(K_{n,n},p))$ or
  $\calA(K_{2n,2n},p)$ as defined in the introduction to this work.

First, we study the behavior of $L(G)$ when $G$ 
  is chosen according $\calS(K_{n,n,p})$.
Recall that in this case, the collection of 
  events $\setof{(x,y),(y,x)}\subseteq E(G)$ are independent, and 
  each one occurs with probability $p$.
Also note that $(x,y)\in E(G)$ if and only if $(y,x)\in E(G)$
  --- any graph for which this equivalence holds will be
  said to be \emph{symmetric}, thus
  motivating the use of the word ``symmetric'' in naming the random graph
  model.
As usual, we begin our study with the determination of 
  the concentration constant for the random model under study.
\begin{lemma}\label{lem:concentration-symm}
The concentration constant for $(\calS(K_{n,n},p))_{n\in\NN}$ is $1/4$.
\end{lemma}
\begin{proof}
Direct application of Talagrand's inequality (as stated 
  in~\cite[Theorem~2.29]{jlr00}).
\end{proof}

As in the study of the binomial model (Section~\ref{sec:random-binomial}) and 
  the word model (Section~\ref{sec:random-word}), given a 
  graph $G$ chosen according to $\calS(K_{n,n},p)$ we will
  consider a reduced graph $G'$ obtained from $G$ by removal of 
  all edges incident to nodes of degree at least $2$.
An important observation is that the graph $G'$ thus obtained 
  is also symmetric.
Since $G'$ is symmetric, the number of vertices of degree $1$ in each 
  of the two color classes of $G'$ must be even, say $2m$.
Thus, the arcs between nodes of degree $1$ in $G'$ can be thought
  of as an involution of $[2m]$ without fix points.
In fact, given that the distribution of $G'$ is invariant under 
  permutation of its nodes, the distribution of $G'$ is also
  invariant under such permutation, and the resulting associated
  involution is distributed as a random involution of $[2m]$
  without fix points.
We shall see that under proper assumptions $L(G)$ 
  and $L(G')$ are essentially equal --- thus, $L(G)$ 
  behaves (approximately) like the length of a longest increasing
  subsequence of a randomly chosen involution of 
  $[2m]$ without fix points.
This partly explains our recollection below of some results
  about the length of a longest increasing subsequence
  of randomly chosen involutions.

Let $\calI_{2m}$ be the distribution of a uniformly 
  chosen involution of $[2m]$ without fix points.
Let $L(\calI_{2m})$ denote the length of the longest 
  increasing subsequence of an involution chosen according 
  to $\calI_{2m}$.
Baik and Rains~\cite{br99} showed that the expected
  value of $L(\calI_{2m})$ is roughly $2\sqrt{2m}$, for $m$ large.
Moreover, Kiwi~\cite[Theorem~5]{} established the following 
  concentration result for $L(I_{2m})$ (we state the 
  result in a weaker form):
\begin{theorem}\label{theo:kiwi}
For $m$ sufficiently large and every $0\leq s\leq 2\sqrt{2m}$,
\begin{eqnarray*}
\prob{}{\left|L(\calI_{2m})-\expec{}{L(\calI_{2m})}\right|
  \geq s+32(2m)^{1/4}} & \leq & 4e^{-s^2/16e^{3/2}\sqrt{2m}}\,.
\end{eqnarray*}
\end{theorem}
\begin{corollary}
For every $0\leq t\leq 1$ and $\alpha>0$ there exists
  a $m_{0}=m_{0}(t,\alpha)$ sufficiently large such that
  for all $m\geq m_{0}$,
\begin{eqnarray*}
\prob{}{\left|L(\calI_{2m})-2\sqrt{2m}\right|\geq 2t\sqrt{2m}}
  & \leq & \alpha\,.
\end{eqnarray*}
\end{corollary}
\begin{proof}
Let $m_{0}=m_{0}(t,\alpha)$ be sufficiently large such that 
  Theorem~\ref{theo:kiwi} and the following 
  conditions hold for all $m>m_{0}$:
\begin{itemize}
\item $\left|\expec{}{L(\calI_{2m})}-2\sqrt{2m}\right|
  + 32(2m)^{1/4}\leq t\sqrt{2m}$.
\item $4e^{-t^{2}\sqrt{2m}/16e^{3/2}} \leq \alpha$.
\end{itemize}
It follows that
\begin{eqnarray*}
\lefteqn{\prob{}{\left|L(\calI_{2m})-2\sqrt{2m}\right| \geq 2t\sqrt{2m}}} \\
  & \leq & \prob{}{\left|L(\calI_{2m})-\expec{}{L(\calI_{2m})}\right|
           \geq 2t\sqrt{2m}-\left|\expec{}{L(\calI_{2m})}-2\sqrt{2m}\right|}
\\
  & \leq & 
  \prob{}{\left|L(\calI_{2m})-\expec{}{L(\calI_{2m})}\right|
              \geq t\sqrt{2m}+32(2m)^{1/4}} 
\\
  & \leq &
  4e^{-t^2\sqrt{2m}/16e^{3/2}}\,.
\end{eqnarray*}
\end{proof}
We now proceed to show that the 
  symmetric random model $\calS(K_{n,n},p)$ admits 
  a $(c,\lambda,\theta)$-median where the constant $c$ 
  is related to a constant that arises in the study of the asymptotic
  behavior of $L(\calI_{2m})$.
We will need the following analogues of Lemmas~\ref{lem:binomial-prep} 
  and~\ref{lem:word-prep}.
\begin{lemma}
Let $n$ be a positive integer.
Let $G$ is chosen according to $\calS(K_{n,n},p)$.
If $E$ and $E'$ denote $E(G)$ and $E(G')$, respectively,
  then
\begin{eqnarray}
\expec{}{|E|} & = & pn(n-1)\,, \label{eqn:symmetric-first} \\
\expec{}{|E'|} & = & pn(n-1)(1-p)^{2n-4}\,, \label{eqn:symmetric-second} \\
\expec{}{|E\setminus E'|} & \leq &
  2p^{2}n(n-1)(n-2)\,. \label{eqn:symmetric-third}
\end{eqnarray}
Moreover, for $\eta>0$,
\begin{eqnarray}
\prob{}{\left||E|-\expec{}{|E|}\right| \geq \eta\expec{}{|E|}}
  & \leq & \frac{2}{\eta^{2}\expec{}{|E|}}\,.
  \label{eqn:symmetric-fourth} 
\end{eqnarray}
\end{lemma}
\begin{proof}{[Sketch]}
For $i\neq j$, let $X_{i,j}$ and $Y_{i,j}$ denote the indicator of the event
  $(i,j)\in E$ and $(i,j)\in E'$, respectively.
Observing that $\expec{}{X_{i,j}}=p$, $\expec{}{Y_{i,j}}=p(1-p)^{2n-4}$,
  $|E|=\sum_{i,j:i\neq j} X_{i,j}$ and 
  $|E'|=\sum_{i,j:i\neq j} Y_{i,j}$, 
  yield~\eqref{eqn:symmetric-first} and~\eqref{eqn:symmetric-second}.
Since $E'\subseteq E$, it follows that
  $|E\setminus E'|=|E|-|E'|$.
Identity~\eqref{eqn:symmetric-fourth} follows from~\eqref{eqn:symmetric-first} 
  and~\eqref{eqn:symmetric-second}
  observing that $(1-p)^{2n-4}\geq 1-(2n-4)p$.

To establish~\eqref{eqn:symmetric-fourth} we observe that 
  $|E|$ can also be expressed as $2\sum_{i<j} X_{i,j}$ and that 
  $\set{X_{i,j}}{i<j}$ is a collection of independent random variables.
To conclude, note that
\[
\Delta \ \ \eqdef \ \ 
  \sum_{(i,j),(k,l):i<j,k<l \atop (i,j)\neq (k,l)} \expec{}{X_{i,j}X_{k,l}}
  \ \ = \ \ {n\choose 2}\left({n \choose 2}-1\right)p^{2} 
  \ \ \leq \ \ \frac{\expec{}{|E|}^{2}}{4}\,,
\]
and apply Chebyshev's inequality for indicator random variables to
  conclude~\eqref{eqn:symmetric-fourth}.
\end{proof}

\begin{proposition}
Let $\delta>0$, $0<p\leq 1$ and $n$ be a positive integer.
There is a sufficiently large constant $C_{1}=C_{1}(\delta)$, 
  and sufficiently small constants $C_{2}$ and $C_{3}$, such
  that
\begin{itemize}
\item 
If $C_{1}/p \leq n^{2} \leq C_{2}\delta/p^{3/2}$, then
  every median of $L(\calS(K_{n,n},p))$ is at most 
  $2(1+\delta)n\sqrt{p}$.

\item 
If $C_{1}/p\leq n^{2}\leq C_{3}\delta^{2}/p^{2}$, then 
  every median of $L(\calS(K_{n,n},p))$ is at least
  $2(1-\delta)n\sqrt{p}$.
\end{itemize}
\end{proposition}
\begin{proof}
Similar to the proof of Proposition~\ref{prop:binomial}.
\end{proof}
We immediately have the following:
\begin{corollary}\label{cor:approximate-median-symm}
The model $(\calS(K_{n,n},p))_{n\in\NN}$ of internal parameter
  $t=1/p$ admits a $(2,1/2,3/4)$-median.
\end{corollary}

We now define an auxiliary distribution which will be useful for
  our study:
\begin{itemize}
\item $\calO(K_{n,n},p)$ 
  (the oriented symmetric binomial random graph model) --- 
  the distribution over the set of 
  subgraphs $H$ of $K_{n,n}$ where 
  the events $\set{H}{(i,j)\in E(H)}$ for $1\leq i<j\leq n$,
  have probability $p$ and are mutually independent, and the 
  events $\set{H}{(i,j)\in E(H)}$, $1\leq j\leq i\leq n$,
  have probability~$0$.
\end{itemize}
(See Figure~\ref{fig:oriented} for an illustration of the distinction
  between distributions $\calS(K_{n,n},p)$ and $\calO(K_{n,n},p)$.)
\begin{figure}\label{fig:oriented}
\begin{center}
\ifpdf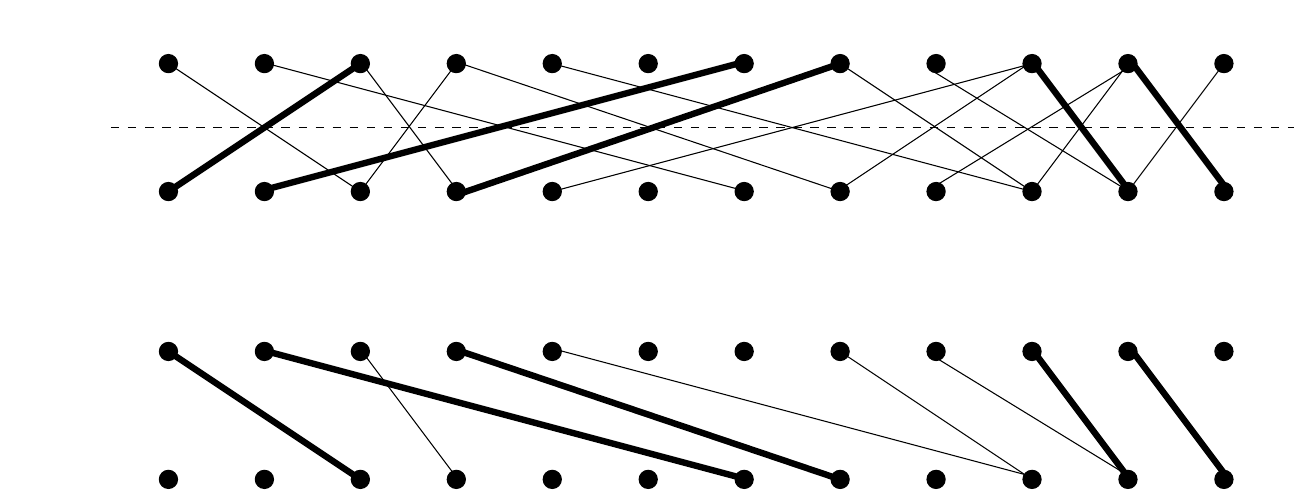\else\input{oriented.pstex_t}\fi
\end{center}
\caption{An illustration of a graph $G$ in the support of 
  $\calS(K_{12,12},p)$ (top), and the graph $O$ in the support 
  of $\calO(K_{12,12},p)$ obtained from $G$ by removal of all
  edges $(x,y)$ such that $x\geq y$ (bottom).
Thicker edges represent a non-crossing matching $M$ of $G$ (top), and 
  the associated non-crossing matching $N$ of $O$ with edge set
  $\set{(\min\setof{x,y},\max\setof{x,y})}{(x,y)\in E(M)}$ (bottom).}
\end{figure}

The following result justifies why we can henceforth work 
  either with $L(\calS(K_{n,n},p))$ or $L(\calO(K_{n,n},p))$.
\begin{lemma}\label{lem:symm-oriented}
The random variables $L(\calS(K_{n,n},p))$ and $L(\calO(K_{n,n},p))$ 
  are identically distributed.
\end{lemma}
\begin{proof}
Let $O$ be a graph in the support of $\calO(K_{n,n},p)$.
We can associate to $O$ a graph $G$
  over the same collection of vertices
  and having edge set $\set{(x,y)}{(x,y)\in E(O) \text{ or } 
    (y,x)\in E(O)}$.
Clearly, $G$ is a symmetric subgraph of $K_{n,n}$ and hence it 
  belongs to the support of $\calS(K_{n,n},p)$.
It is easy to see that the mapping from $O$ to $G$ is one-to-one.
Moreover, the probability of $G$ being chosen under $\calS(K_{n,n},p)$
  is exactly equal to the probability of occurrence of $O$
  under $\calO(K_{n,n},p)$.

On the other hand, if $M$ is a non-crossing subgraph of $G$, then
  there is a non-crossing subgraph of $O$ (and hence of $G$),
  say $N$, whose size is the same as the one of $M$.
Indeed, it suffices to take as the collection of edges of $N$ the 
  set $\set{(\min\setof{x,y},\max\setof{x,y})}{(x,y)\in E(M)}$.
(See Figure~\ref{fig:oriented} for an illustration of 
  the relation between $M$ and $N$.)
We get that $L(G)=L(O)$, which concludes the proof.
\end{proof}

We are now ready to prove the main result of this section.
\begin{theorem}\label{th:random-symm}
For every $\epsilon>0$ there is a sufficiently small constant 
  $p_0$ and a sufficiently large constant $A$ such that 
  for all $p\leq p_{0}$ and $n\geq A/\sqrt{p}$,
\begin{equation}\label{eqn:mainsymm-expec}
(1-\epsilon)2n\sqrt{p} 
  \ \ \leq \ \
  \expec{}{L(\calS(K_{n,n},p))}
  \ \ \leq \ \
  (1+\epsilon)2n\sqrt{p}\,,
\end{equation}
and the following hold
\begin{itemize}
\item If $\median{}{L(\calS(K_{n,n},p))}$ is a median of 
  $L(\calS(K_{n,n},p))$,
\begin{equation}\label{eqn:mainsymm-median}
(1-\epsilon)2n\sqrt{p} 
  \ \ \leq \ \
  \median{}{L(\calS(K_{n,n},p))}
  \ \ \leq \ \
  (1+\epsilon)2n\sqrt{p}\,.
\end{equation}
\item There is an absolute constant $C>0$, such that
\begin{eqnarray}
\prob{}{L(\calS(K_{n,n},p))\leq (1-\epsilon)2n\sqrt{p}} 
  & \leq & \exp\left(-C\epsilon^{2}n\sqrt{p}\right)\,, 
\label{eqn:mainsymm-lower} \\
\prob{}{L(\calS(K_{n,n},p))\geq (1+\epsilon)2n\sqrt{p}} 
  & \leq & \exp\left(-C\frac{\epsilon^{2}}{1+\epsilon}n\sqrt{p}\right)\,.
\label{eqn:mainsymm-upper}
\end{eqnarray}
\end{itemize}
\end{theorem}
\begin{proof}
Unfortunately, $(\calS(K_{n,n},p))_{n\in\NN}$ is not a random 
  hyper-graph model, so we can not immediately apply the Main Theorem.
However, it is a weak random hyper-graph model.
Hence, to prove the lower bound in~\eqref{eqn:mainsymm-expec}
  and~\eqref{eqn:mainsymm-median}, and inequality~\eqref{eqn:mainsymm-lower},
  we use the fact that the model $\calS(K_{n,n},p)$ with internal
  parameter $t=1/p$ has
  a concentration constant $h=1/4$ (Lemma~\ref{lem:concentration-symm})
  admits a $(2,1/2,3/4)$-median (Corollary~\ref{cor:approximate-median-symm}),
  and apply the Main Theorem.

To prove the remaining bounds, consider a bipartite 
  graph $H$ chosen according to $\calG(K_{n,n},p)$, and let 
  $O$ be the graph obtained from $H$ by deletion of 
  all its edges $(x,y)$ such that $x\geq y$.
Since $O$ is a subgraph of $H$, it immediately follows that 
  $L(O)\leq L(H)$.
Note that $O$ follows the distribution $\calO(K_{n,n},p)$.
By Lemma~\ref{lem:symm-oriented}, 
  $L(O)$ has the same distribution as $L(\calS(K_{n,n},p)$.
Hence, if $n$ and $p$ satisfy the hypothesis of 
  Theorem~\ref{th:random-binomial}
\begin{eqnarray*}
\expec{}{L(\calS(K_{n,n},p))} 
  &  = &
  \expec{}{L(O)}
  \ \ \leq \ \ 
  \expec{}{L(H)}
  \ \ \leq \ \ 
  (1+\epsilon)2n\sqrt{p}\,, \\
\median{}{L(\calS(K_{n,n},p))} 
  & = &
  \median{}{L(O)}
  \ \ \leq \ \ 
  \median{}{L(H)}
  \ \ \leq \ \ 
  (1+\epsilon)2n\sqrt{p}\,,
\end{eqnarray*}
and provided $C$ is as in Theorem~\ref{th:random-binomial},
\begin{eqnarray*}
\lefteqn{\prob{}{L(\calS(K_{n,n},p))\geq (1+\epsilon)2n\sqrt{p}}
  \ \ =  \ \
  \prob{}{L(O)\geq (1+\epsilon)2n\sqrt{p}}} \\
  & \leq & 
  \prob{}{L(H)\geq (1+\epsilon)2n\sqrt{p}}
  \ \ \leq \ \ 
  \exp\left(-C\frac{\epsilon^{2}}{1+\epsilon}2n\sqrt{p}\right)\,.
\end{eqnarray*}
This concludes the proof of the stated result.
\end{proof}

We can now establish Theorem~\ref{th:main-random-symm}.
\begin{prooff}{Theorem~\ref{th:main-random-symm}}
Let $n,n',n''$ be positive integers such that $n=n'+n''$.
Clearly, 
\begin{eqnarray*}
\expec{}{L(\calS(K_{n,n},p))} & \geq &
  \expec{}{L(\calS(K_{n',n'},p))}+\expec{}{L(\calS(K_{n'',n''},k))}\,.
\end{eqnarray*}
By subadditivity, it follows that the limit of 
  $\expec{}{L(\calS(K_{n,n},p))}$ when normalized by $n$ 
  exists and equals $\sigma_{p}=\inf_{n\in\NN}\expec{}{L(\calS(K_{n,n},p))/n}$.
A direct application of Theorem~\ref{th:main-random-symm}
  yields that $\sigma_{p}/\sqrt{p}\to 2$ when $p\to 0$.
\end{prooff}

One can also show, although not as straightforward as for the 
  case of the symmetric binomial random graph model, that the following 
  analogue of Theorem~\ref{th:random-symm}
  holds for the anti-symmetric case.
\begin{theorem}\label{th:random-anti-symm}
For every $\epsilon>0$ there is a sufficiently small constant 
  $p_0$ and a sufficiently large constant $A$ such that 
  for all $p\leq p_{0}$ and $n\geq A/\sqrt{p}$,
\begin{equation}\label{eqn:mainantisymm-expec}
(1-\epsilon)4n\sqrt{p} 
  \ \ \leq \ \
  \expec{}{L(\calA(K_{2n,2n},p))}
  \ \ \leq \ \
  (1+\epsilon)4n\sqrt{p}\,,
\end{equation}
and the following hold
\begin{itemize}
\item If $\median{}{L(\calA(K_{2n,2n},p))}$ is a median of 
  $L(\calA(K_{2n,2n},p))$,
\begin{equation}\label{eqn:mainantisymm-median}
(1-\epsilon)4n\sqrt{p} 
  \ \ \leq \ \
  \median{}{L(\calA(K_{2n,2n},p))}
  \ \ \leq \ \
  (1+\epsilon)4n\sqrt{p}\,.
\end{equation}
\item There is an absolute constant $C>0$, such that
\begin{eqnarray}
\prob{}{L(\calA(K_{2n,2n},p))\leq (1-\epsilon)4n\sqrt{p}} 
  & \leq & \exp\left(-C\epsilon^{2}n\sqrt{p}\right)\,, 
\label{eqn:mainantisymm-lower} \\
\prob{}{L(\calA(K_{2n,2n},p))\geq (1+\epsilon)4n\sqrt{p}} 
  & \leq & \exp\left(-C\frac{\epsilon^{2}}{1+\epsilon}n\sqrt{p}\right)\,.
\label{eqn:mainantisymm-upper}
\end{eqnarray}
\end{itemize}
\end{theorem}
\begin{proof}
Omitted from current draft.
\end{proof}
Theorem~\ref{th:main-random-anti-symm} can now be established
  much in the same way as Theorem~\ref{th:main-random-symm} was 
  derived.

\bibliographystyle{alpha}
\bibliography{biblio}

\end{document}